\documentclass{article}

\usepackage[english]{babel}
\usepackage{graphicx} 
\usepackage{float} 
\usepackage{fullpage}
\usepackage{amsfonts}
\usepackage{amsmath}
\usepackage{amssymb}
\usepackage{amsbsy}
\usepackage{amsthm}
\usepackage{mathtools}
\usepackage{epsfig}
\usepackage{graphicx}
\usepackage{subfig}
\usepackage{colordvi}
\usepackage{graphics}
\usepackage{color}
\usepackage{cite}
\usepackage{bm}
\usepackage{verbatim} 
\numberwithin{equation}{section}
\usepackage{xcolor}
\usepackage{enumitem}

\newtheorem{thm}{Theorem}[section]

\newtheorem{lemma}[thm]{Lemma}

\newtheorem{alg}[thm]{Algorithm}
\newtheorem{remark}[thm]{Remark}
\newtheorem{assumption}[thm]{Assumption}

\newcommand{\argmin}[1]{\underset{#1}{\operatorname{arg}\,\operatorname{min}}\;}

\usepackage[letterpaper,top=2cm,bottom=2cm,left=3cm,right=3cm,marginparwidth=1.75cm]{geometry}

\usepackage{amsmath}
\usepackage{graphicx}
\usepackage[colorlinks=true, allcolors=blue]{hyperref}

\title{NGMRES convergence analysis and proof of acceleration for contractive and noncontractive iterations}

\author{
Yunhui He\thanks{Department of Mathematics, University of Houston, 3551 Cullen Blvd, Room 641, Houston, Texas 77204-3008, USA (yhe43@central.uh.edu).}
\and 
Leo G Rebholz\thanks{School of Mathematical and Statistical Sciences, Clemson University, Clemson, SC 29634, USA (rebholz@clemson.edu).  The work of LR was partially funded by Department of Energy grant DE-SC0025292.}
\and 
Mengying Xiao\thanks{Department of Mathematics and Statistics, University of West Florida, Pensacola, FL 32514, USA (mxiao@uwf.edu).}
}

\begin{document}
\maketitle

\begin{abstract}
This paper gives the first convergence analysis and proof of acceleration for nonlinear GMRES (NGMRES) applied to contractive and noncontractive fixed point iterations (FPIs) for solving general nonlinear systems.  Our main results are that in both the contractive and noncontractive cases, the ratio gain of the optimization problem used to determine the combination coefficients of  NGMRES update is the mechanism responsible for accelerating (or enabling) convergence.  Our analysis also reveals a second important quantity related to the optimization problem, which directly predicts the linear convergence rate at each iteration.
Numerical results for several challenging nonlinear test problems are presented to  illustrate the theory. The results show how the acceleration improves convergence, show that the quantity predicting the linear convergence rate is remarkably accurate and moreover can be useful for adaptively choosing NGMRES depth, show how restarts can improve convergence in noncontractive iterations, show how NGMRES is naturally suited for finding distinct solutions of a multi-solution PDE, and that NGMRES can perform better than Anderson acceleration when applied to superlinear FPIs.
\end{abstract}

\section{Introduction}

Nonlinear GMRES (NGMRES) was first proposed in 1997 by Washio and Oosterlee as 
a Krylov subspace based acceleration technique for solving difficult nonlinear equations \cite{washio1997krylov}.  Shortly after, they extended the approach to nonlinear multigrid and to solve recirculating incompressible flow problems  \cite{oosterlee2000krylov}.  There have been many works since then using NGMRES to accelerate or enable convergence in nonlinear solvers for many types of application problems, including image restoration \cite{chang2003acceleration,savage2005improved,he2025generalized}, tensor decomposition 
\cite{sterck2012nonlinear, sterck2013steepest, sterck2021asymptotic}, Navier-Stokes equations \cite{HR26}, and CUTEst problems \cite{riseth2019objective}.  NGMRES takes the following form:  for a given fixed point operator $q(x)$ used to solve $g(x)=0$, NGMRES with depth $m$ at iteration $k+1$ can be written as 
\begin{equation}
 				x_{k+1} = q(x_k) + \sum_{i=0}^{m_k}\beta_i^{k+1} \left(q(x_k)- x_{k-i} \right), \label{eq:min-NG1} 
\end{equation}
where $m_k=\min\{k,m\}$, and coefficient $\beta^{k+1}=\big(\beta_0^{k+1},\beta_1^{k+1},\cdots, \beta_{m_k}^{k+1}\big)$ (note superscripts refer to the iteration number) are determined by solving the  least-squares optimization problem
 \begin{equation}\label{eq:min-NG2} 
 				\min_{\beta^{k+1}} \| g(q(x_k))+\sum_{i=0}^{m_k} \beta_i^{k+1} \left(g(q(x_k))-g(x_{k-i}) \right)\|^2.
 \end{equation}
NGMRES can equivalently be written in a constrained optimization form, which is more natural for our analysis herein:
\[
 				x_{k+1} = \alpha^{k+1}_{k+1} q(x_k) + \sum_{j=k-m_k}^k \alpha^{k+1}_i x_i,
\]
\[
\mbox{ where }
 				\alpha^{k+1} = \argmin{\sum_{i=k-m_k}^{k+1} \alpha_i^{k+1} = 1} 
				\| \alpha^{k+1}_{k+1} g(q(x_k))+ \sum_{j=k-m_k}^k \alpha^{k+1}_i g(x_i) \|^2.
\]
While the norm historically used in the NGMRES optimization problems above is most often $\ell^2$, the convergence theory in this paper requires it to be the range norm of $g$.  Numerical tests in this paper and \cite{HR26} show that while sometimes similar convergence behavior is observed with these two norms, other times using the range norm of $g$ produces significantly better results.

A convergence theory and proof of acceleration for NGMRES was recently established in \cite{HR26} for NGMRES applied to the specific problem of solving the steady Navier-Stokes equations (NSE) with the Picard iteration, and under a small data assumption 
to enforce that Picard is contractive with rate $\kappa_{nsepic}<1$ for both the fixed point residuals and the nonlinear residuals \cite{HR26}.  To our knowledge, this was the first work to mathematically prove that NGMRES accelerates convergence, albeit for a specific contractive iteration and nonlinear system.

The purpose of this paper is to generalize and extend the NSE/Picard-specific results to general nonlinear systems $g=0$ where $g:H\rightarrow H'$ (where $H$ is a Hilbert space and $H'$ is the dual space of $H$) with associated fixed point iteration function $q:H\rightarrow H$.  We provide herein a complete single step convergence analysis for NGMRES applied to FPIs for nonlinear systems, revealing {\it how} it accelerates and even enables convergence.  Our theory covers both the contractive and noncontractive cases, as well as general depths $m\ge 0$.  

The results we find, under some reasonable smoothness assumptions on $g$ and $q$ (given in section 2) and denoting convergence ratio of the nonlinear residuals for the unaccelerated FPI to be $\kappa$ (which the theory allows to be greater than 1), take the form:
\begin{align}
\| g(x_{k+1}) \|_{H'} & \le \min\{ \gamma_{k+1} \kappa,\theta_{k+1} \} \| g(x_{k})   \|_{H'} + \mbox{ higher order terms}, 
 \end{align}
where $\theta_{k+1}$ and $\gamma_{k+1}$ are defined by
\begin{align*}
 \theta_{k+1} & := \frac{ \| \alpha^{k+1}_{k+1} g( q(x_k)) + \alpha^{k+1}_k g(x_k) + ... + \alpha^{k+1}_{k-m}g(x_{k-m}) \|_{H'}  }{ \| g(x_k) \|_{H'} }, \\
 \gamma_{k+1}& := \frac{ \| \alpha^{k+1}_{k+1} g(q(x_k) ) + \alpha^{k+1}_k g(x_k) + ... + \alpha^{k+1}_{k-m}g(x_{k-m}) \|_{H'}  }{ \| g(q(x_k)) \|_{H'} }.
 \end{align*}
 Note that both $\theta_{k+1}\le 1$ and $\gamma_{k+1}\le 1$ due to the structure of the optimization problem, and are only equal to one in the special case that the (constrained form) optimization problem returns $(0,1,0,...,0)$ or $(1,0,...,0)$, respectively. The only difference between $\theta_{k+1}$ and $\gamma_{k+1}$ is the denominator in their respective definitions.  The term $\gamma_{k+1}$ is observed to be the `gain of the optimization problem,' as it is the ratio of what the optimization objective function gives to what a FPI with no acceleration gives.   A key point is that $\gamma_{k+1}$ scales the linear convergence rate $\kappa$ of the nonlinear residuals to $\gamma_{k+1}\kappa$, which identifies the gain of the optimization problem as the mechanism responsible for acceleration by NGMRES.  Hence up to higher order terms, this result shows that NGMRES accelerates convergence precisely through the gain of the optimization problem.  The term $\theta_{k+1}$  in the convergence results directly approximates the linear convergence rate and moreover, since $\theta_{k+1}\le 1$ (and only equals 1 in very special cases), it establishes convergence of NGMRES, up to the effect of higher order terms.  Numerical tests in Section 4 show that $\theta_{k+1}$ is a remarkably accurate predictor of the convergence rate when the higher order terms are negligible, and thus it can be used to determine when the nonlinear contributions are negligible and thus determine an adaptive depth strategy to speed up convergence even more. 

To our knowledge, these are the first theoretical results for NGMRES that prove acceleration and identify the mechanism responsible for it, for both contractive and noncontractive iterations, for any depth, and for general nonlinear systems.  Furthermore, we show that $\theta_{k+1}$ is a very accurate approximation of the nonlinear residuals' linear convergence rate and that it can be used to develop an adaptive depth strategy, how restarts can significantly improve convergence in the noncontractive case, and that better convergence can be obtained if the optimization norm matches that of the theory.  Existing theory for NGMRES is all quite recent, and includes important results from  \cite{he2025convergence} where it is established the NGMRES converges $q$-linearly (and $r$-linearly when $m=0$). Moreover, full depth NGMRES applied to the Richardson iteration is equivalent to classical GMRES \cite{greif2026convergence}, an extension of this result to preconditioned Richardson and preconditioned GMRES in  \cite{he2026ngmresprecon}, and convergence analysis for alternating NGMRES is given for linear systems in \cite{he2026convergenceANG}.  Hence we believe the results herein establishing NGMRES acceleration for general nonlinear systems and associated contractive and noncontractive iterations help to fill a gap in NGMRES theory.

NGMRES is related to the recently popularized Anderson acceleration (AA) \cite{Anderson65,PR25} method in that it has a similar extrapolation that defines the next iterate, but it significantly differs in that the optimization problem of NGMRES uses nonlinear residuals while AA uses fixed point residuals.  Although for many problems AA and NGMRES provide similar improvement to convergence behavior \cite{HR26b,sterck2021asymptotic}, NGMRES can outperform AA in highly anisotropic problems, steep optimization landscapes, and systems with sharp discontinuities \cite{sterck2012nonlinear,sterck2021asymptotic,he2025generalized}.  Still, comparisons in the literature are scarce and it is not yet possible to draw any strong conclusions.  It is an open question to classify what types of applications and iterations will have one of AA or NGMRES significantly outperform the other.  We partially address this question herein, by showing that NGMRES is better suited for use with superlinearly converging iterations.  

We note that convergence results for AA applied to general FPIs were established in \cite{PR21,PR25} that 
are similar to NGMRES results herein as they show the gain of the AA optimization problem is the mechanism responsible for the acceleration provided by AA, and AA adds higher order terms to the fixed point residual expansion.  A key difference compared to NGMRES results is that AA theory is fixed-point residual based, while NMGRES theory is based on the true nonlinear residuals.  Another key difference in our NGMRES theory is that for equal depths, higher order terms for AA are lower order than for NGMRES; in particular for depth 1, NGMRES has quadratic higher order terms while AA's higher order terms are of order 1.62 \cite{RX23}.  Hence NGMRES appears better suited for use with superlinearly converging methods, and we give a numerical test that illustrates this.  Yet another difference in AA and NGMRES theories is the simplicity of the NGMRES proofs: while the proofs that establish AA theory are rather long and technical \cite{PR25,PR21,EPRX20,PRX19, X23, RX23}, the NGMRES convergence and acceleration theory herein is much shorter and straight-forward.

For the choice of the optimization norm, we find analogous results as found for AA in \cite{HR25}: {\it using the correct norm in the optimization problem matters}.  For many years,  $\ell^2$ was the most common choice for AA and NGMRES optimization norms since it was efficient and simple to use, and there was no theory yet that suggested to use a different norm.  However, this is no longer true.  The NGMRES theory herein requires the range norm of $g$ be used for the optimization problem, and tests herein and in \cite{HR26} show that using $\ell^2$ instead can give worse (or no) convergence. AA theory relies on the Hilbert space norm associated with $q$ being used for the AA optimization problem, and in \cite{HR25} it is shown that using this norm instead of $\ell^2$ can provide better convergence results.  

This paper is arranged as follows.  In section 2 we give some mathematical preliminaries to set the notation, present assumptions, and recall known results to allow for a smooth analysis to follow.  Section 3 gives a convergence analysis for NGMRES for contractive and noncontractive fixed point iterations, and also shows how NGMRES provides acceleration.  Section 4 gives results of several numerical tests which illustrate the new theory and show the effectiveness of the method.

\section{Mathematical Preliminaries}
We now present notation and mathematical preliminaries to be used in the sections to follow.  Let  $(H,\|\cdot \|_H)$ be a Hilbert space with inner product $(\cdot,\cdot)$, and $(H', \|\cdot \|_{H'})$ be the dual space of $H$ with $\displaystyle \|y\|_{H'} = \sup\limits_{0\neq x\in H} \frac{\langle y,x\rangle}{\|x\|}$, where $\langle \cdot,\cdot \rangle$ represents the duality pairing. 

This paper is concerned with solving nonlinear systems of the form $g(x)=0$ for $g:H\rightarrow H'$ and $f\in H'$, where $x^*\in H$ denotes a root of  $g(x)=0$.  The associated fixed point function to find the root is denoted by $q:H\rightarrow H$, and the goal of NGMRES is to accelerate convergence of the FPI $x_{k+1}=q(x_k)$, or even enable convergence when the FPI fails to converge.  We make the following regularity assumptions on $g$ and $q$, and fix notation for the associated constants.
Together, these assumptions amount to requiring $g$ to be Lipschitz continuously differentiable on $H$, and $q$ to be continuous and satisfy that the fixed point residuals cannot grow beyond the nonlinear residuals without bound.  Despite the weakened assumptions on $q$ relative to $g$,  many fixed point functions $q$ are derived from $g$ (e.g. $q(x)=g(x)+x$) so the assumed regularity on $g$ is  inherited by $q$.  
A key feature of our analysis is that it covers both the contractive and noncontractive cases.

\begin{assumption}
\label{assum1}
    Let $g:  H  \to H'$ be a nonlinear operator, $x^*$ solve $g(x) =0$, and $x^*$ be a fixed point of $q:H\rightarrow H$.
        We make the following assumptions: 
    \begin{enumerate}
        \item $g$ is Lipschitz continuously Fr\'echet differentiable:
        \begin{enumerate}
        \item $g':H\times H \to H'$ satisfies, for $x,h\in H$, 
        \begin{align*}
      \lim\limits_{\|h\|_H \to 0}  \frac{\| g(x+h) - g(x) - g'(x;h)\|_{H'}}{\|h\|} =0.
    \end{align*} 
	\item There exists $\sigma_1>0$ satisfying for all $x,y,h\in H$,
    \begin{align}
      \| g'(x;h)- g'(y;h)\|_{H'} \le& \sigma_1 \|x-y\|_{H}\|h\|_{H}.
      \label{gineq1}
    \end{align}
 \end{enumerate}
 
     \item  $q$ is continuous and there exists a constant $\kappa_0>0$ such that for all $x\in H$,
    \begin{align}
    \label{qbd1}
         \|  q(x) - x \|_{H} \le \kappa_0 \| g(x)\|_{H'}.
    \end{align}
    \item  There exists a constant $\kappa_1>0$ such that for all $x\in H$,
    \begin{align}
    \label{qbd2}
         \|  g(q(x)) \|_{H'} \le \kappa_1 \| g(x)\|_{H'}.
    \end{align}
    \end{enumerate}  
\end{assumption}

\begin{remark}
The Picard iteration for the steady NSE studied in \cite{PRX19,HR26} satisfies these conditions if $q$ is the Picard iteration solution operator and $g$ is the nonlinear residual operator, see \cite{HR26}.  For a simple root, one can in general expect that Assumptions 2.1.2 and 2.1.3 will hold locally.  That is, if a FPI converges, then the nonlinear residuals of those iterates should converge as well, and vice versa.
\end{remark}

\begin{remark}
Assumption 2.1.2 enforces that the fixed point residuals must converge to zero if the nonlinear residuals do.  Equation \eqref{qbd2} establishes $\kappa_1$ as the convergence ratio of the nonlinear residuals in the case of no acceleration.  
\end{remark}

Note that since $g$ is Fr\'echet differentiable, by definition $g'$ is linear with respect to the second argument.
Consequently we have
   \begin{align}
   \label{dglinear}
       g'(\cdot; t_1h_1+ t_2h_2) = t_1g'(\cdot ; h_1) + t_2 g'(\cdot, h_2),\ \  \forall t_1,t_2\in \mathbb{R}, h_1,h_2\in H.
   \end{align} 
 
We now recall the depth $m$ NGMRES algorithm
for accelerating a fixed-point iteration defined by $q(x)$ that is  used to solve $g(x)=0$.  Here we use the constrained form of the optimization problem, which is equivalent to the more commonly used formulation \eqref{eq:min-NG1}-\eqref{eq:min-NG2}.  
 
\begin{alg}[NGMRES]
\label{alg:ngmres}
Consider NGMRES at Step $k+1$ with $k\geq 0$ with depth $m=\min \{ m_{max},k \}$, and $x_0\in H$ to be an initial guess. Then
\begin{enumerate}
    \item [1.] Find the coefficients $(\alpha^{k+1}_{k+1}, \alpha^{k+1}_{k}, \cdots, \alpha^{k+1}_{k-m})$ satisfying 
    \begin{align}
       \min\limits_{\sum\limits_{i=k-m}^{k+1} \alpha_i^{k+1} =1} \left\| \alpha_{k+1}^{k+1} g(q(x_k)) + \sum\limits_{j=k-m}^k \alpha^{k+1}_j g(x_j) \right\|^2_{H'}.
       \label{minstep}
    \end{align}
    \item [2.] Set $x_{k+1} =\alpha_{k+1}^{k+1} q(x_k) + \sum\limits_{j=k-m}^k \alpha^{k+1}_j x_j $.
\end{enumerate}
\end{alg}
The depth $m$ can vary at each step, and $m$ can be no larger than $k$.  While it is common to use a constant depth $m$ after the initial $m$ iterations, we discuss an adaptive strategy for dynamically picking $m$ in Section 4.

While historically the $\ell^2$ norm is commonly used for the optimization norm, our results herein rely on using the $H'$ norm.  Moreover, in \cite{HR26} it is shown that for $g$ representing the NSE and $q$ the associated Picard iteration, using the dual norm for the optimization problem gives much better results compared to using $\ell^2$ for certain problems, in particular in 3D.

We make the following assumption on the uniform boundedness of the $\{ \alpha_j^k \}$ coefficients. 
\begin{assumption}
\label{assum3}
    We assume $\sum\limits_{j=k-m}^{k+1} |\alpha_j^{k+1}|\le \bar \alpha$ for some $\bar\alpha>0$.
\end{assumption}
This assumption is implied by the optimization problems being (uniformly numerically) well-posed.  
Consequently, we have \begin{align}
1 = \sum_{j=k-m}^{k+1} \alpha_j^{k+1} \le \sum_{j=k-m}^{k+1} |\alpha_j^{k+1}| \le \bar\alpha.
\label{alphabd}
\end{align}
Note that this can be checked on the fly, and $m$ can be reduced or filtering can be performed if $\bar\alpha$ is deemed too large. Further details and strategies for enforcing this assumption can be found in \cite{PR23,PR25,K18}.

We define the following important quantities, which arise in our convergence analysis.
\begin{align}
\label{eqn:theta}
    \theta^m_{k+1}\coloneqq \frac{\left\|\alpha_{k+1}^{k+1} g(q(x_k)) + \sum\limits_{j=k-m}^k \alpha^{k+1}_j g(x_j) \right\|_{H'} }{\|g(x_k)\|_{H'}}, \\
    \label{eqn:gamma}
    \gamma^m_{k+1}  \coloneqq \frac{\left\|\alpha_{k+1}^{k+1} g(q(x_k)) + \sum\limits_{j=k-m}^k \alpha^{k+1}_j g(x_j) \right\|_{H'}}{\|g(q(x_k))\|_{H'}}
\end{align}
From  \eqref{minstep}, it holds that $0\le \gamma^m_{k+1},\theta_{k+1} \le 1$.  Observe that the numerator of $\gamma^m_{k+1}$ is the optimization objective function evaluated at the minimizer $\alpha^{k+1}=\{ \alpha^{k+1}_i \}_{i=k-m}^{k+1}$, and the denominator is the optimization objective function evaluated at $(1,0,...,0)$.  
Hence $\gamma^m_{k+1}=1$ only
when $\alpha^{k+1}=(1,0,...,0)$ solves the optimization problem, and so except in these very rare instances it holds that  $0\le \gamma^m_{k+1}<1$ (and our numerical tests illustrate this).  Similarly, $\theta^m_{k+1} = 1$ only if $\alpha^{k+1}=(0,1,0,...,0)$.  Thus, we can generally expect $\gamma^m_{k+1},\theta_{k+1} < 1$ except in diabolical cases.  Note that if one does not perform NGMRES and simply does a fixed point iteration, this is equivalent to running NGMRES with optimization coefficient $\alpha^{k+1}=(1,0,...,0)$.  Hence we refer to $\gamma^m_{k+1}$ as the ratio `gain of the optimization problem' at Step $k+1$.

\section{Convergence}

We now present a new convergence theory for NGMRES.  This is a single step analysis, and we set $m=\min \{m_{max},k \}$.  We first consider the base case of $m=0$, and prove a convergence result that holds for contractive and noncontractive fixed point iterations.  We then prove an analogous result for general $m\ge 1$ in the contractive case, and finally a result for $m\ge 1$ in the noncontractive case.  In each case the main result is that NGMRES scales the linear convergence rate of the nonlinear residuals (i.e. $\kappa_1$) by the gain of the optimization problem $\gamma^m_{k+1}$, and $\theta^m_{k+1} \le 1$ directly predicts the NGMRES nonlinear residuals' linear convergence rate (up to higher order terms). 

\subsection{NGMRES convergence theory for $m=0$}

Consider Algorithm \ref{alg:ngmres} with $m=0$, which reduces to the following algorithm.
\begin{alg}[NGMRES $m=0$]
\label{alg:ngmres0}
Let $x_0$ be an initial guess in $H$.  Then at Step $k+1$,
\begin{enumerate}
    \item[1.] Find $\alpha^{k+1}$ such that 
    \begin{align}
        \min\limits_{\alpha^{k+1}} \left\| \alpha^{k+1} g(q(x_k)) + (1-\alpha^{k+1}) g(x_k)\right\|_{H'}.
    \end{align}
    \item[2.] Set $x_{k+1} = \alpha^{k+1} q(x_k) + (1-\alpha^{k+1})x_k$.
\end{enumerate}
\end{alg}

\begin{thm}[NGMRES convergence with depth $m=0$]
\label{thm:ngmres0}
    Let Assumptions \ref{assum1} and \ref{assum3} hold. Then iterate $x_{k+1}$ from Algorithm \ref{alg:ngmres0} satisfies
     \begin{align}
     \label{eq:gx10}
        \| g(x_{k+1})\|_{H'} \le \theta^0_{k+1} \|g(x_k)\|_{H'} + \mathcal{O}(\|g(x_k)\|_{H'}^2), \\
         \| g(x_{k+1})\|_{H'} \le \gamma^0_{k+1} \kappa_1 \|g(x_k)\|_{H'} + \mathcal{O}(\|g(x_k)\|_{H'}^2).
         \label{eq:gx20}
    \end{align}
\end{thm}

We recall the standard mean-value/Fundamental Theorem of Calculus formula along the segment from $a$ to $b$:  
$$f(b)-f(a)=\int_0^1 f'(a+t(b-a);b-a)\,dt=\int_0^1 f'(b+t(a-b);b-a)\,dt,$$
which is very useful for our proofs in the following.

\begin{remark} The `gain of the optimization problem'  $\gamma^0_{k+1}$ is shown in the theorem to scale the linear convergence rate $\kappa_1$ of the nonlinear residuals.  Since it is expected that $\gamma^0_{k+1} < 1$ except in diabolical cases where it can equal 1, $\gamma^0_{k+1}$ represents the mechanism responsible for NGMRES convergence acceleration.
\end{remark}
\begin{remark}
The term $\theta_{k+1}^0$ is shown in the theorem to approximate the linear convergence rate of NGMRES nonlinear residuals.  Our numerical tests indicate that it is remarkably accurate in this regard, once the nonlinear residuals are sufficiently small so that higher order terms are negligible.
\end{remark}
\begin{remark}
The higher order terms in the nonlinear residual bounds are quadratic, suggesting that using NGMRES with convergence order
$1<r\le 2$ convergent fixed-point iterations will preserve the order.  This is in contrast to depth 1 AA, where the higher order terms are less than quadratic and can decelerate superlinearly convergence \cite{X23,RX23}.  This feature of NGMRES (over AA) is illustrated in Section 4.4.
\end{remark}

\begin{remark}
The results of the above theorem cover the results established in \cite{HR26}, where NGMRES is applied to solve the steady NSE with the Picard iteration.
\end{remark}

\begin{proof}
Expanding $x_{k+1}$, we can write 
\begin{align*}
    x_{k+1} - q(x_k) =& \alpha^{k+1} q(x_k) + (1-\alpha^{k+1})x_k - q(x_k)
    =  -(1-\alpha^{k+1})(q(x_k) - x_k), \\
    x_{k+1} - x_k = & \alpha^{k+1} q(x_k) + (1-\alpha^{k+1})x_k - x_k
    = \alpha^{k+1} (q(x_k) - x_k).
\end{align*}
By the triangle inequality and definition of $\theta_{k+1}^0$, we have
\begin{align*}
   \| g(x_{k+1}) \|_{H'}
    \le  & \| \alpha^{k+1} g( q(x_k))  + (1-\alpha^{k+1}) g(x_k) \|_{H'} + \bigg\| g(x_{k+1}) - \alpha^{k+1} g( q(x_k)) - (1-\alpha^{k+1})g(x_k) \bigg\|_{H'} 
    \\
    = & \theta_{k+1}^0 \| g(x_k)\|_{H'}+ \bigg\| \alpha^{k+1} \left( g(x_{k+1}) -  g( q(x_k)) \right)
    + (1-\alpha^{k+1}) \left( g(x_{k+1})- g(x_k)  \right)  \bigg\|_{H'} .
\end{align*}
It remains to bound the last term to verify \eqref{eq:gx10}. Applying Taylor expansion, the expanding $x_{k+1}$ results above, and \eqref{dglinear} of $g'$ produces
\begin{align*}
   \| & \alpha^{k+1}  \left( g(x_{k+1}) -  g( q(x_k)) \right)
    + (1-\alpha^{k+1}) \left( g(x_{k+1})- g(x_k)  \right)  \|_{H'}
    \\
       & = \bigg\| \alpha^{k+1} \int_0^1 g'(z_{k_1}(t);  x_{k+1}-q(x_k) ) dt 
    + (1-\alpha^{k+1}) \int_0^1 g'( z_{k_2}(t); x_{k+1}-x_k ) dt \bigg\|_{H'}
    \\
    & = \bigg\| \alpha^{k+1} \int_0^1 g'(z_{k_1}(t);  -(1-\alpha^{k+1})(q(x_k)-x_k)) dt 
    + (1-\alpha^{k+1}) \int_0^1 g'( z_{k_2}(t); \alpha^{k+1}(q(x_k)-x_k) ) dt \bigg\|_{H'}
    \\
    & =| \alpha^{k+1}(1-\alpha^{k+1}) | \left\| \int_0^1\bigg( g'(z_{k_2}(t); q(x_k)-x_k)  - g'(z_{k_1}(t); q(x_k)-x_k)    \bigg)  dt  \right\|_{H'},
       \end{align*}
       where $z_{k_1}(t) = x_{k+1} - t(x_{k+1} - q(x_k)) , z_{k_2}(t) = x_{k+1} - t (x_{k+1} - x_k) $.
       Then applying \eqref{gineq1}, \eqref{qbd1}, Assumption \ref{assum3} and \eqref{eqn:theta}, we obtain
        \begin{align*}
    | \alpha^{k+1}(1-\alpha^{k+1})& | \left\| \int_0^1\bigg( g'(z_{k_2}(t); q(x_k)-x_k)  - g'(z_{k_1}(t); q(x_k)-x_k)    \bigg)  dt  \right\|_{H'}
  \\
    \le  & (1+\bar \alpha ) \bar \alpha  \int_0^1\left\| g'(z_{k_2}(t); q(x_k)-x_k)  - g'(z_{k_1}(t); q(x_k)-x_k)  \right\|_{H'} \ dt
    \\
    \le &  (1+\bar \alpha ) \bar \alpha \sigma_1  \int_0^1  \|z_{k_2}(t) - z_{k_1}(t)\|_{H}\|q(x_k)-x_k\|_{H}  dt
    \\
    \le &  (1+\bar \alpha ) \bar \alpha \sigma_1 \|q(x_k)- x_k\|_{H}^2 \int_0^1 |t| dt \\
    \le &  \frac12 (1+\bar \alpha ) \bar \alpha \sigma_1 \kappa_0^2 \|g(x_k)\|_{H'}^2.
\end{align*}
This proves \eqref{eq:gx10}.  

 For \eqref{eq:gx20}, we utilize the same techniques as above along with \eqref{qbd2} in \eqref{eq:gx10} as
  \begin{align*}
      \| \alpha^{k+1} g(q(x_k)) + (1-\alpha^{k+1})g(x_k) \|_{H'} = \gamma^0_{k+1} \| g(q(x_k))\|_{H'} \le \gamma_{k+1}^0 \kappa_1 \| g(x_k)\|_{H'}.
  \end{align*} 
  This completes the proof.
\end{proof}

\subsection{NGMRES convergence theory for general $m$: the contractive case}

We now extend the NGMRES convergence and acceleration results to the case of general $m$.  Here, we split the analysis into two cases: contractive and noncontractive.  While for $m=0$ our results hold for both the contractive and noncontractive fixed-point iterations, for $m\ge 1$ the structure of NGMRES leads us to different analyses for these cases.  We now give results for the contractive case, after a preliminary lemma.

\begin{lemma} 
\label{lemma0}
Suppose $q$ is Lipschitz continuous on $H$ with Lipschitz constant $\kappa_q<1$.  Then for $i,j = k-m, k-m+1, \dots, k$ and $i>j$, we have
    \begin{align*}
        \| q(x_k) - x_j\|_{H} \le& \kappa_0\|g(x_k)\|_{H'} + \frac{\kappa_0}{1-\kappa_q}\sum\limits_{n=j}^{k}\|g(x_{n})\|_{H'},\\
        \|x_j - x_i\|_{H} \le&  \frac{\kappa_0}{1-\kappa_q}\sum\limits_{n=j}^{i}\|g(x_{n})\|_{H'}.
    \end{align*}
\end{lemma}
\begin{proof}
We first bound $x_{n+1} -x_n$, where $n= k-m, k-m+1, \cdots, k-1.$  Using that $q$ is Lipschitz continuous and applying \eqref{qbd1}, the triangle inequality produces
    \begin{align*}
        \| x_{n+1} - x_n\|_{H} 
        \le & \|x_{n+1} - q(x_{n+1})\|_{H} + \| q(x_{n+1}) - q(x_n)\|_{H}+ \|q(x_n) - x_n\|_{H}
        \\
        \le & \kappa_0 \|g(x_{n+1})\|_{H'} + \kappa_q\|x_{n+1} - x_n\|_{H} + \kappa_0\|g(x_n)\|_{H'}.
        \end{align*}
Rearranging gives 
\begin{align*}
    \|x_{n+1} - x_n\|_{H} \le \frac{\kappa_0}{1-\kappa_q} (\|g(x_{n+1})\|_{H'}+\|g(x_{n})\|_{H'}).
\end{align*}
Consequently, we have 
   \begin{align*}
        \| x_i - x_j\|_{H} \le &  \sum\limits_{n=j}^{i-1} \| x_{n+1} - x_n\|_{H}
        \le \frac{\kappa_0}{1-\kappa_q}\sum\limits_{n=j}^{i}\|g(x_{n})\|_{H'}, \text{ for } i>j, \mbox{ and }
    \\
        \| q(x_k) - x_j\|_{H} \le &\|q(x_k) - x_k\|_{H} + \|x_k -x_j\|_{H}
        \le \kappa_0\|g(x_k)\|_{H'} + \frac{\kappa_0}{1-\kappa_q}\sum\limits_{n=j}^{k}\|g(x_{n})\|_{H'},
    \end{align*}
    thanks to triangle inequality.  This completes the proof.
\end{proof}

\begin{thm}[NGMRES convergence with general $m$ (contractive case)]
\label{thm:ngmres}
  Let Assumptions \ref{assum1} and \ref{assum3} hold, and suppose $q$ is Lipschitz continuous on $H$ with Lipschitz constant $\kappa_q<1$.  Then the solution $x_{k+1}$ from Algorithm \ref{alg:ngmres} with $m\ge1$ satisfies     
     \begin{align}
     \label{eq:gx1mC}
        \| g(x_{k+1})\|_{H'} \le \theta^m_{k+1} \|g(x_k)\|_{H'} + C\bar\alpha^2  \sum\limits_{j=k-m}^k\|g(x_j)\|_{H'}^2, \\
         \| g(x_{k+1})\|_{H'} \le \gamma^m_{k+1} \kappa_1 \|g(x_k)\|_{H'} +  C\bar\alpha^2   \sum\limits_{j=k-m}^k\|g(x_j)\|_{H'}^2,
         \label{eq:gx2mC}
    \end{align}
    where $C$ is constant depending on $\kappa_0, \kappa_q, \sigma_1,  m$.
    
\end{thm}
\begin{remark} Just as in the $m=0$ case, the gain of the optimization problem  $\gamma^m_{k+1}$ is revealed to be the mechanism responsible for NGMRES convergence acceleration.  The term $\theta_{k+1}^m$ approximates the linear convergence rate of NGMRES, and we show in Section 4
that it is remarkably accurate in this regard.  We also show in Section 4 that $\theta_{k+1}^m$ can be used to construct an adaptive strategy for choosing the NGMRES depth $m$ at each step.
\end{remark}

\begin{proof}
For $j = k-m, k-m+1, \dots, k,$  we compute 
\begin{align*}
     x_{k+1} -x_j 
      = &  \alpha^{k+1}_{k+1} q(x_k) + \sum\limits_{i= k-m}^k \alpha_i^{k+1}x_i  -\bigg( \alpha_{k+1}^{k+1} +  \sum\limits_{i = k-m}^k \alpha_i^{k+1} \bigg) x_j\\
  = & \alpha_{k+1}^{k+1} (q(x_k) - x_j) + \sum\limits_{i=k-m}^k \alpha_i^{k+1} (x_i - x_j),
\end{align*}
and \begin{align*}
    x_{k+1} - q(x_k) = & 
   - \sum\limits_{i=k-m}^k \alpha_i^{k+1} (q(x_k) - x_i).
\end{align*}
thanks to $\alpha_{k+1}^{k+1}+ \sum\limits_{i=k-m}^k \alpha_i^{k+1} =1$.
From the triangle inequality, we have
   \begin{align}
     &   \hspace{-.3in} \| g(x_{k+1})\|_{H'} 
     \nonumber\\
    \le  &\left\| \alpha_{k+1}^{k+1} g( q(x_k)) + \sum\limits_{j=k-m}^k \alpha^{k+1}_j g(x_k) \right\|_{H'}  + 
    \bigg\| g(x_{k+1}) - \alpha^{k+1}_{k+1} g( q(x_k)) - \sum\limits_{j=k-m}^k \alpha^{k+1}_jg(x_j) \bigg\|_{H'} 
    \nonumber \\
    \le &  \theta_{k+1}^m \|g(x_{k} )\|_{H'} +   \bigg\| \alpha_{k+1}^{k+1} \bigg(g (x_{k+1}) - g(q(x_k)) \bigg)
    + \sum\limits_{j=k-m}^k \alpha_j^{k+1}\bigg( g(x_{k+1})  - g(x_j) \bigg) \bigg\|_{H'} .
    \label{eq:temp}
      \end{align}
      It remains to bound the last term of this last expression to obtain \eqref{eq:gx1mC}. From Taylor expansion, expansion of $x_{k+1}$ and \eqref{dglinear}, we rearrange to get
      \begin{align*}
     g (x_{k+1}) - g(q(x_k)) =  &  \int_0^1 g'\left(z_{k+1,k+1}(t); x_{k+1}-q(x_k) \right) dt
      \\
    = &  \int_0^1 g'\left(z_{k+1,k+1}(t); -\sum\limits_{i=k-m}^k \alpha_i^{k+1} (q(x_k) - x_i)\right) dt 
      \\
      = & -\sum\limits_{i=k-m}^k \alpha_i^{k+1}  \int_0^1 g'(z_{k+1,k+1}(t); q(x_k)-x_i) dt ,
      \end{align*}   
      where $ z_{k+1,k+1}(t) =  x_{k+1} -t (x_{k+1}-q(x_k)).$ 
 Similarly, for $j = k-m, k-m+1, \dots, k,$ we have
      \begin{align*}
      g(x_{k+1})  - g(x_j)  = &  \int_0^1 g'\left(z_{k+1,j}(t); \left(  \alpha_{k+1}^{k+1} q(x_k) + \sum\limits_{i=k-m}^k \alpha^{k+1}_i x_i \right) - x_j\right) dt
      \\
      =  & \int_0^1 g'\left(z_{k+1,j}(t);   \alpha_{k+1}^{k+1} (q(x_k) - x_j) + \sum\limits_{i=k-m}^k \alpha^{k+1}_i (x_i - x_j) \right)  dt
	\\
      = & \alpha_{k+1}^{k+1} \int_0^1 g'(z_{k+1,j}(t); q(x_k) - x_j) +  \sum\limits_{i = k-m}^k \alpha_i^{k+1}\int_0^1 g'(z_{k+1,j}(t); x_i - x_j) dt,
      \end{align*}
      where $ z_{k+1, j}(t) = x_{k+1}t(x_{k+1}-x_j).$ 
Combining the above two equations yields
\begin{align*}
   &  \alpha_{k+1}^{k+1} \bigg(g (x_{k+1}) - g(q(x_k)) \bigg)
    + \sum\limits_{j=k-m}^k \alpha_j^{k+1}\bigg( g(x_{k+1})  - g(x_j) \bigg)
    \\
    =&  -\alpha_{k+1}^{k+1} \sum\limits_{i=k-m}^k \alpha_i^{k+1}  \int_0^1 g'(z_{k+1,k+1}(t); q(x_k)-x_i) dt 
    +\alpha_{k+1}^{k+1} \sum\limits_{j=k-m}^k \alpha_j^{k+1} \int_0^1 g'(z_{k+1,j}(t); q(x_k)-x_j)dt \\
     &  + \sum\limits_{j=k-m}^k \alpha_j^{k+1} \sum\limits_{i = k-m}^k \alpha_i^{k+1}\int_0^1 g'(z_{k+1,j}(t); x_i - x_j) dt
        \\
      = & \alpha_{k+1}^{k+1} \sum\limits_{j=k-m}^k \alpha_j^{k+1}  \int_0^1 g'(z_{k+1,j}(t); q(x_k)-x_j) - g'(z_{k+1,k+1}(t), q(x_k)-x_j) dt 
     \\
      &  + \sum_{i, j=k-m, i>j}^k \alpha_i^{k+1} \alpha_j^{k+1}
      \int_0^1 g'(z_{k+1,j}(t); x_j - x_i) - g'(z_{k+1, i}(t) ; x_j -x_i) dt,
      \end{align*}
where the last step comes from
       \begin{align*}
     &  \sum_{i, j=k-m }^k \alpha_i^{k+1} \alpha_j^{k+1}
      \int_0^1 g'(z_{k+1,j}(t); x_j - x_i) \ dt 
      \\ 
      = &
       \sum_{i, j=k-m, i>j }^k \alpha_i^{k+1} \alpha_j^{k+1}
      \int_0^1 g'(z_{k+1,j}(t); x_j - x_i) \ dt  +  \sum_{i, j=k-m , i<j}^k \alpha_i^{k+1} \alpha_j^{k+1}
      \int_0^1 g'(z_{k+1,j}(t); x_j - x_i) \ dt 
      \\
      = &
       \sum_{i, j=k-m, i>j }^k \alpha_i^{k+1} \alpha_j^{k+1}
      \int_0^1 g'(z_{k+1,j}(t); x_j - x_i) \ dt  +  \sum_{i, j=k-m , i>j}^k \alpha_i^{k+1} \alpha_j^{k+1}
      \int_0^1 g'(z_{k+1,i}(t); x_i - x_j) \ dt 
      \\ = &
        \sum_{i, j=k-m, i>j}^k \alpha_i^{k+1} \alpha_j^{k+1}
      \int_0^1  g'(z_{k+1,j}(t); x_j - x_i) - g'(z_{k+1, i}(t) ; x_j -x_i) dt.
      \end{align*}
Applying  \eqref{gineq1}, the last term in \eqref{eq:temp} reduces with elementary analysis to
\begin{align}
 & \left\| \alpha_{k+1}^{k+1} \bigg(g (x_{k+1}) - g(q(x_k)) \bigg)
    + \sum\limits_{j=k-m}^k \alpha_j^{k+1}\bigg( g(x_{k+1})  - g(x_j) \bigg) \right\|_{H'}
    \nonumber \\
    \le &    |\alpha_{k+1}^{k+1}| \sum\limits_{j=k-m}^k |\alpha_j^{k+1}|  \int_0^1  \left\|g'(z_{k+1,j}(t); q(x_k)-x_j) - g'(z_{k+1,k+1}(t), q(x_k)-x_j) \right\|_{H'} dt 
    \nonumber \\ &
     +  \sum\limits_{i,j=k-m, i>j}^k |\alpha_j^{k+1}||\alpha_i^{k+1}| \int_0^1 \left\| g'(z_{k+1,j}(t); x_j - x_i) - g'(z_{k+1, i}(t); x_j - x_i)  \right\|_{H'}dt 
     \nonumber \\
     \le & \sigma_1 |\alpha_{k+1}^{k+1}| \sum\limits_{j=k-m}^k |\alpha_j^{k+1}| \| q(x_k) - x_j\|_{H} \int_0^1\|z_{k+1,j}(t) - z_{k+1,k+1}(t)\|_{H} dt
    \nonumber \\&
    + \sigma_1\sum\limits_{i,j = k-m, i>j}^k |\alpha_j^{k+1}||\alpha_i^{k+1}| \|x_j - x_i\|_{H} \int_0^1 \| z_{k+1, j}(t) - z_{k+1,i}(t)\|_{H}dt 
    \nonumber \\
    \le & 
\frac{1}{2}\sigma_1 |\alpha_{k+1}^{k+1}| \sum\limits_{j=k-m}^k |\alpha_j^{k+1}| \| q(x_k) - x_j\|_{H}^2
    + \frac{1}{2}\sigma_1\sum\limits_{i,j = k-m, i>j}^k |\alpha_j^{k+1}||\alpha_i^{k+1}| \|x_j - x_i\|_{H}^2. \label{P1}
\end{align}
From here, applying Lemma \ref{lemma0} produces \eqref{eq:gx1mC}.
 
To prove \eqref{eq:gx2mC}, begin with \eqref{eq:temp} and repeat the analysis above except use the definition of $\gamma_{k+1}^m$ along with \eqref{qbd2} in \eqref{eq:temp},
  \begin{align*}
      \left\| \alpha_{k+1}^{k+1} g(q(x_k)) + \sum\limits_{j= k-m}^k \alpha_j^k g(x_k)  \right\|_{H'} = \gamma^m_{k+1} \| g(q(x_k))\|_{H'} \le \gamma^m_{k+1} \kappa_1 \| g(x_k)\|_{H'}.
  \end{align*} 
  This finishes the proof.
  \end{proof}

\subsection{NGMRES convergence theory for general $m$: the noncontractive case}

Now we consider the general depth case in the noncontractive case. Here, we use a different preliminary lemma.
\begin{lemma} 
\label{lemma1}
Let $m\ge 1$, then for any nonnegative integer $n$, the sequence of iterates $\{ x_n \}$ from Algorithm \ref{alg:ngmres} satisfies
\begin{align}
 \| x_{k+1} - x_k\| \le \sum\limits_{j=0}^{k} \widetilde C_{k-j}^m \bar\alpha^{k-j+1} \|g(x_j)\|_{H'},
 \label{ineq:induc}
\end{align}
where $\{ \widetilde C_{j}^m \}$ are constants depending on $j,m$.
\end{lemma}
\begin{proof}
We use induction to establish the bound on  $\{ \|x_{k+1} - x_k\|_H \}_{k=0}^\infty$.
With an initial guess $x_0$ in Algorithm \ref{alg:ngmres}, we obtain 
\begin{align*}
 \|x_1  - x_0\|_{H} =&  \| \alpha^1_1 q(x_0) + \alpha_0^1 x_0 - x_0\|_{H} 
 \le  | \alpha^1_1| \| q(x_0) - x_0\|_{H} \le \bar \alpha \|g(x_0)\|_{H'},
\end{align*}
 where $\alpha_0^1+ \alpha_1^1 =1$. 
 Similarly,
 \begin{align*}
 \| x_2 - x_1\|_{H}  = & \| \alpha_2^2 q(x_1) + \alpha_1^2 x_1 + \alpha_0^2 x_0 - x_1\|_{H} 
 \le |\alpha_2^2| \| q(x_1)-x_1\|_H + |\alpha_0^2 | \| x_1-x_0\|_{H}
\\
  \le& \bar \alpha \| q(x_1)-x_1\|_H + \bar \alpha \| x_1-x_0\|_{H}.
 \end{align*}
At Step $k+1$, from the expansion of $x_{k+1}$ and \eqref{qbd1} we have
 \begin{align*}
 \| x_{k+1} - x_k\|_{H}  = & \left\| \alpha^{k+1}_{k+1} q(x_k) + \sum\limits_{j= k-m}^k \alpha_j^{k+1} x_j  - x_k \right\|_{H}
 \\
 \le & | \alpha^{k+1}_{k+1}|  \| q(x_k)  - x_k \|_{H}+ \sum\limits_{j= k-m}^{k-1} |\alpha_j^{k+1}| \| x_k - x_j\|_{H}
 \\
 \le & \bar \alpha \|g(x_k)\|_{H'} + \bar \alpha \sum\limits_{j=k-m}^{k-1} \sum\limits_{i = j }^{k-1} \| x_{i+1} - x_{i}\| _{H}
 \\
  = &  \bar \alpha \| g(x_k)\|_{H'} + \bar\alpha \sum_{j=1}^{m} (m-j+ 1) \|x_{k-j+1} - x_{k-j}\|_H.
 \end{align*}
Next, we apply the method of generating functions to verify \eqref{ineq:induc}. 
Denoting $ d_k = \|x_{k+1} - x_k\|_H,\, g_k = \| g(x_k)\|_{H'},$ and $\omega_k = m-k+1,$
we rewrite the above recurrence inequalities as
\begin{align*}
d_k \le \bar\alpha g_k + \bar\alpha \sum\limits_{j=1}^{m} \omega_j d_{k-j}, 
\end{align*}
where $k = 0, 1,2,\dots $ and $d_k =0, $ if $ k\le  0$.
Multiplying the equation by $t^k$ and summing from $k=0$ to $\infty$ gives
\begin{align}
\label{Dineq}
D(t) \le  &  \bar\alpha G(t)  + \bar\alpha \sum_{k=0}^\infty \bigg( \sum_{j=1}^m \omega_j d_{k-j} \bigg) t^k 
=  \bar\alpha G(t)  +  \bar\alpha \bigg( \sum_{j=1}^m \omega_j t^j \bigg) \bigg( \sum_{j=0}^\infty d_j t^j \bigg) =  \bar\alpha G(t)  +  \bar\alpha \Omega_m(t) D(t),
\end{align}
where $
D(t) = \sum\limits_{j=0}^\infty d_j t^j, \quad 
G(t) = \sum\limits_{j=0}^\infty g_j t^j, \quad 
\Omega_m(t) = \sum\limits_{j=1}^m \omega_j t^j$.
It is clear that $\Omega_m(t)$ is a polynomial function with $\Omega_m(0) =0$.  
By the continuity, there exists $r = r(\bar\alpha, m)>0$ such that $\Omega_m(t) <\frac1{\bar\alpha}$ for all $|t|< r$. 
Letting $|t|<r$, 
the power series of $\frac{1}{1-\bar\alpha \Omega_m(t)}$ is bounded by 
\begin{align*}
\frac{1}{1-\bar\alpha \Omega_m(t)}=& 
\sum\limits_{n=0}^\infty \bar\alpha^n \left(\sum_{j=1}^m w_jt^j \right)^n
= \sum\limits_{n=0}^\infty \bar\alpha^n \bigg( \sum_{j=n}^{nm} C(j,n,m) t^j\bigg)
\\
=& \sum\limits_{j=0}^\infty \bigg( \sum_{n=1}^j C(j,n,m)\bar\alpha^n \bigg) t^j 
= \sum\limits_{j=0}^\infty \bigg( \sum_{n=1}^j C(j,n,m)\bar\alpha^{n-j} \bigg) \bar\alpha^j t^j\
 \\
 \le & \sum\limits_{j=0}^\infty \bigg( \sum_{n=1}^j C(j,n,m) \bigg) \bar\alpha^j t^j 
 \coloneqq  \sum\limits_{j=0}^\infty  \widetilde C_j^m \bar\alpha^j  t^j,
\end{align*}
thanks to \eqref{alphabd},
where $\widetilde C_j^m$ depends on the integers $j,m$, and $\{\widetilde C_j^m \bar\alpha^j\}$ are upper bounds of the coefficients of the power series. 
Thus \eqref{Dineq} reduces to
\begin{align*}
D(t) \le &  \frac{\bar\alpha }{1-\bar\alpha \Omega_m(t)} G(t)
\le  \bar\alpha \sum\limits_{j=0}^\infty \widetilde C_j^m\bar\alpha^{j}  t^j \sum\limits_{j=0}^\infty g_j t^j
  = \bar\alpha  \sum\limits_{k=0}^\infty \bigg( \sum\limits_{j=0}^{k} \widetilde C_{k-j}^m  \bar\alpha^{k-j} g_j\bigg)  t^k ,
\end{align*}
 Consequently,  
\begin{align*}
\|x_{k+1} - x_k\|_H \le \bar\alpha \sum\limits_{j=0}^k \widetilde C_{k-j}^m \bar\alpha^{k-j}\|g(x_j)\|_{H'}.
\end{align*}
 This completes the proof.
\end{proof}
\begin{thm}[NGMRES convergence with general $m$ (noncontractive case)]
\label{thm:ngmres2}
  Let Assumptions \ref{assum1} and \ref{assum3} hold.  Then the solution $x_{k+1}$ from Algorithm \ref{alg:ngmres} with $m\ge1$ satisfies 
     \begin{align}
     \label{eq:gx1m}
        \| g(x_{k+1})\|_{H'} \le \theta^m_{k+1} \|g(x_k)\|_{H'} +  \sum\limits_{j=0}^k \widetilde C^m_{\bar\alpha}(j)\|g(x_j)\|_{H'}^2, \\
         \| g(x_{k+1})\|_{H'} \le \gamma^m_{k+1} \kappa_1 \|g(x_k)\|_{H'} +   \sum\limits_{j=0}^k \widetilde C^m_{\bar\alpha}(j)\|g(x_j)\|_{H'}^2,
         \label{eq:gx2m}
    \end{align}   
     where $\widetilde C^m_{\bar\alpha}(j)= \left\{ \begin{array}{ll}
     C_{\bar\alpha}^{m}(j), & \bar\alpha>1 \\ 
      C(j), & \bar\alpha=1 \end{array} \right.$, $C_{\bar\alpha}^m(j) =  \left\{ \begin{array}{ll} C\bar\alpha^2 , & j=k, \\[8pt]
		C\bar\alpha^{4}\frac{1-\bar\alpha^{2(k-j)}}{1-\bar\alpha^2}, & j = k-m,\dots, k-1,\\[8pt]     
      C\bar\alpha^{2(k-m-j)+4}\frac{1-\bar\alpha^{2m}}{1-\bar\alpha^2}, & j = 0,1,\dots, k-m-1. \end{array}\right. $ and  $C$ is constant depending on $m,\sigma_1, \kappa_0$.
\end{thm}
\begin{remark}\label{rmk:con-noncon} 
Note that Theorem 3.11 holds for contractive iterations as well, although the contractive result from Theorem \ref{thm:ngmres} is a stronger result.  An important difference between the contractive and noncontractive results, i.e. Theorems \ref{thm:ngmres} and \ref{thm:ngmres2}, is that the higher order terms retain their dependence on early iterations. Moreover, the coefficients of the higher order terms in Theorem \ref{thm:ngmres} depend on $\bar\alpha^2$, while the ones in Theorem \ref{thm:ngmres2} are scaled by $\bar\alpha^{2k}$, and due to  \eqref{alphabd}, Theorem \ref{thm:ngmres} is believed a sharper result for the contractive case. 

The bounds in Theorem \ref{thm:ngmres2} suggest higher order terms can become a dominant error source and prevent convergence as $k$ gets larger.  Hence restarts of NGMRES can be very important in the noncontractive case. In our numerical tests, we show restarts dramatically help convergence for the nonlinear Helmholtz problem.
\end{remark}

\begin{proof}
We begin the proof of Theorem \ref{thm:ngmres2} identically to that of Theorem \ref{thm:ngmres}, all the way to \eqref{P1}:
\begin{align*}
 & \left\| \alpha_{k+1}^{k+1} \bigg(g (x_{k+1}) - g(q(x_k)) \bigg)
    + \sum\limits_{j=k-m}^k \alpha_j^{k+1}\bigg( g(x_{k+1})  - g(x_j) \bigg) \right\|_{H'}
    \\
    \le & 
  \frac12 \sigma_1 |\alpha_{k+1}^{k+1}| \sum\limits_{j=k-m}^k |\alpha_j^{k+1}| \| q(x_k) - x_j\|_H^2 
    + \frac12 \sigma_1 \sum\limits_{i,j = k-m, i>j}^k |\alpha_j^{k+1}||\alpha_i^{k+1}| \|x_j - x_i\|_H^2.
\end{align*}
From here, applying $
 \| x_i - x_j\|_H \le   \sum\limits_{n=j}^{i-1} \| x_{n+1} - x_n\|_H, 
 \ \ 
        \| q(x_k) - x_j\|_H \le \|q(x_k) - x_k\|_H + \|x_k -x_j\|_H,$
Lemma \ref{lemma1} and Assumption \ref{assum3} to the above inequality, 
we obtain
\begin{align*}
 & \left\| \alpha_{k+1}^{k+1} \bigg(g (x_{k+1}) - g(q(x_k)) \bigg)
    + \sum\limits_{j=k-m}^k \alpha_j^{k+1}\bigg( g(x_{k+1})  - g(x_j) \bigg) \right\|_{H'}
    \\
\le & \sigma_1 \bar\alpha^2 (m+1) \|q(x_k)- x_k\|_H^2
+C \sigma_1 \bar\alpha^2 \sum\limits_{j=k-m}^{k-1} \sum_{n=j}^{k-1} \|x_{n+1} - x_n\|_H^2
  + C \sigma_1 \bar\alpha^2 \sum_{i,j=k-m, i>j}^k \sum_{n=j}^{i-1} \|x_{n+1} - x_n\|_H^2
  \\
  \le& 
  \sigma_1 \bar\alpha^2 (m+1) \kappa_0^2 \|g(x_k)\|_{H'}^2
  + \sigma_1\bar\alpha^2 C \sum_{n=k-m}^{k-1} \|x_{n+1}-x_n\|_H^2
  \\
  \le& 
  \sigma_1 \bar\alpha^2 (m+1) \kappa_0^2 \|g(x_k)\|_{H'}^2
  + 2\sigma_1\bar\alpha^4C  \sum_{n=k-m}^{k-1} \sum_{j=0}^n \widetilde C_j^m \bar\alpha^{2(n-j)} \|g(x_j)\|_{H'}^2
  \\
  \le& C\bar\alpha^2  \|g(x_k)\|_{H'}^2 
  + C \bar\alpha^4 \sum_{j=0}^{k-m-1}\left( \sum_{n=k-m}^{k-1} \bar\alpha^{2(n-j)} \right) \|g(x_j)\|_{H'}^2
   + C\bar\alpha^4 \sum_{j=k-m}^{k-1} \left( \sum_{n = j }^{k-1} \bar\alpha^{2(n-j)} \right) \|g(x_j)\|_{H'}^2
  \\
  \le& 
  \left\{ \begin{array}{ll} 
  C\bar\alpha^2  \|g(x_k)\|_{H'}^2 
  + C \bar\alpha^4 \sum_{j=0}^{k-m-1} \bar\alpha^{2(k-m-j)}\frac{1-\bar\alpha^{2m}}{1-\bar\alpha^2} \|g(x_j)\|_{H'}^2
  + C \bar\alpha^4\sum_{j=k-m}^{k-1} \frac{1-\bar\alpha^{2(k-j)}}{1-\bar\alpha^2} \| g(x_j)\|_{H'}^2, & \bar\alpha >1,\\[10pt]
   C \sum_{j=0}^{k}\|g(x_j)\|_{H'}^2, & \bar\alpha = 1,
  \end{array}
  \right.
\end{align*}
where $C$ represents some positive constant independent of $\bar\alpha$.
Then combining \eqref{eq:temp} produces \eqref{eq:gx1m}. Finally, \eqref{eq:gx2m} is established similarly except using \eqref{qbd2}.
\end{proof}

\section{Numerical Experiments}

We now give results for four numerical tests that illustrate the theory above for NGMRES.  These tests show the effectiveness of NGMRES in accelerating and enabling convergence, show that $\theta_k^m$ is a very accurate approximation of the nonlinear residuals' linear convergence rate and how this can be used to develop an adaptive depth strategy, how restarts can significantly improve convergence in the noncontractive case, and that better convergence can be obtained if the optimization norm matches that of the theory.  We also show how NGMRES is well suited to use with known deflation techniques to directly find distinct solutions for multi-solution PDEs, and that NGMRES is effective for superlinearly convergent iterations.  Unless otherwise mentioned, $H'$ is used as the optimization norm in our tests.

\subsection{Nonlinear Helmholtz equation}

The nonlinear Helmholtz (NLH) equation arises in nonlinear optics as a
model for propagation of linearly-polarized, time-harmonic electromagnetic waves in Kerr-type dielectrics. i.e. ``lasers.''  Following \cite{BFT07,FT01}, the one dimensional NLH system can be written using two-way boundary conditions as: Find $u:[0,1]\rightarrow \mathbb{C}$ satisfying
\begin{align*}
\frac{d^2 u}{dx^2} + k_0^2 \left( 1 + \epsilon(x) |u|^2 \right) u & = 0, \ \ \ 0<x<1, \\
\frac{du}{dx} + i k_0 u & = 2 i k_0,\ \ \ x=0,\\
\frac{du}{dx} - i k_0 u & = 0,\ \ \ x=1.
\end{align*}
Here, $k_0$ represents the linear wave number in the surrounding medium, and we take $\epsilon(x)$ on $[0,1]$ to be a piecewise constant function to approximate a grated Kerr medium by
\[
\epsilon(x) = \left\{  \begin{array}{lc} 0 & 0\le x<\frac13 \\ 0.5 & \frac13\le x < \frac23 \\ 1 & \frac23 \le x\le 1 \end{array} \right. .
\]
Solutions for $k_0$=20, 40 and 60 (found using our methods below) are shown in Figure \ref{NLH1}.
\begin{figure}[ht!]
\begin{center}
\includegraphics[width = .32\textwidth, height=.32\textwidth,viewport=0 0 530 430, clip]{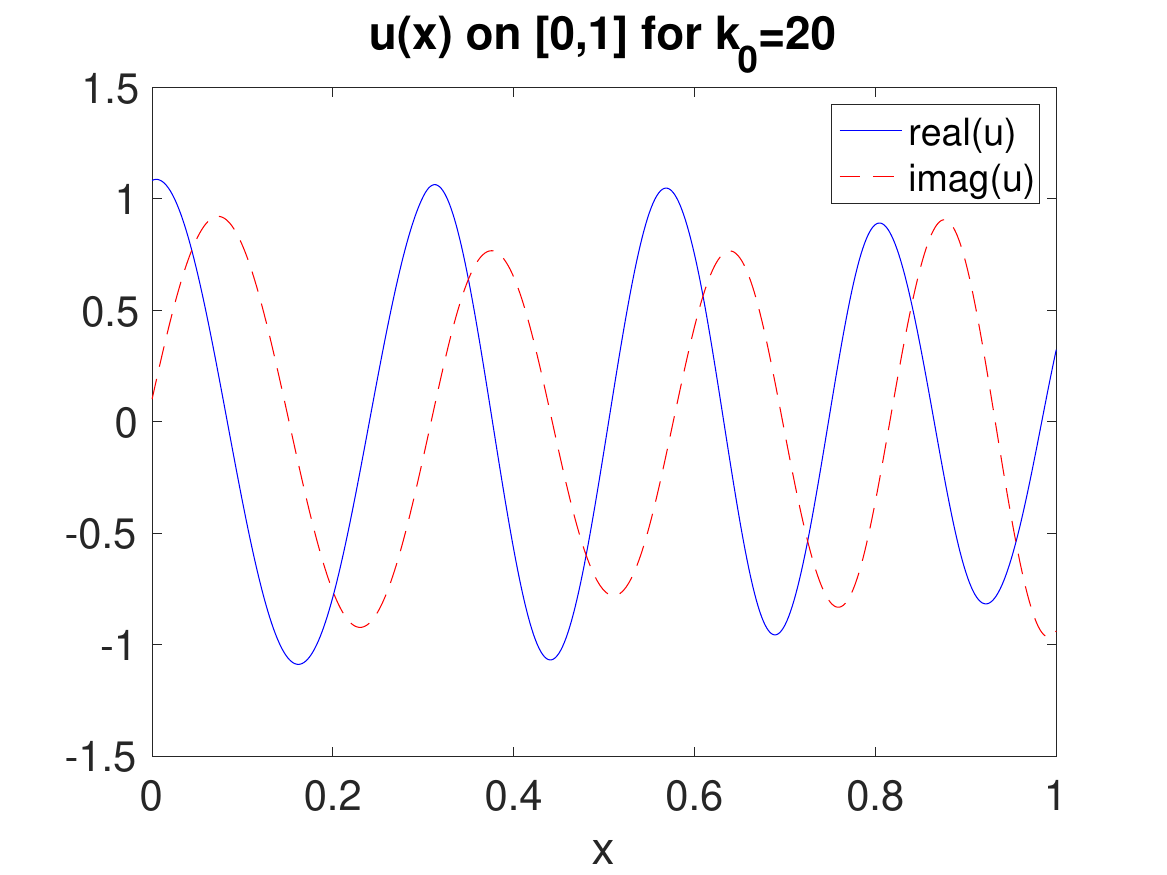}
\includegraphics[width = .32\textwidth, height=.32\textwidth,viewport=0 0 530 430, clip]{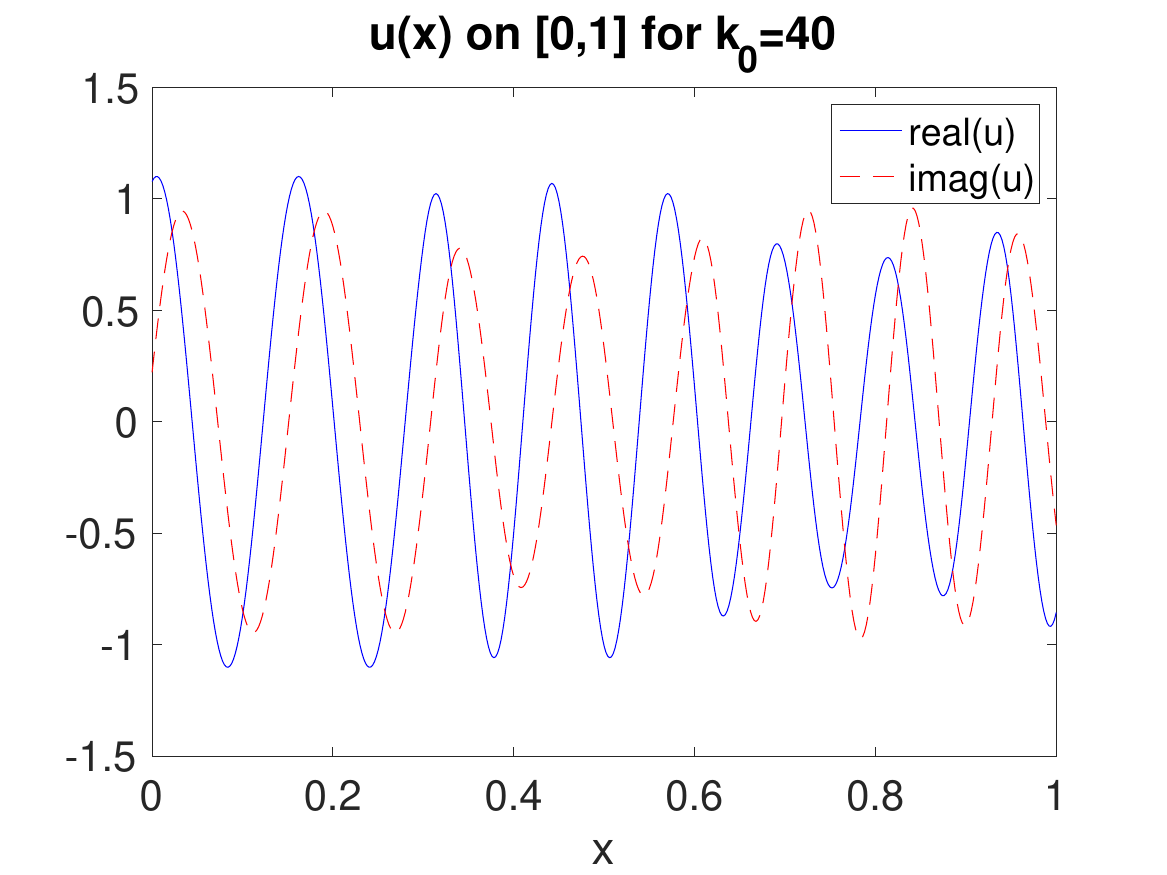}
\includegraphics[width = .32\textwidth, height=.32\textwidth,viewport=0 0 530 430, clip]{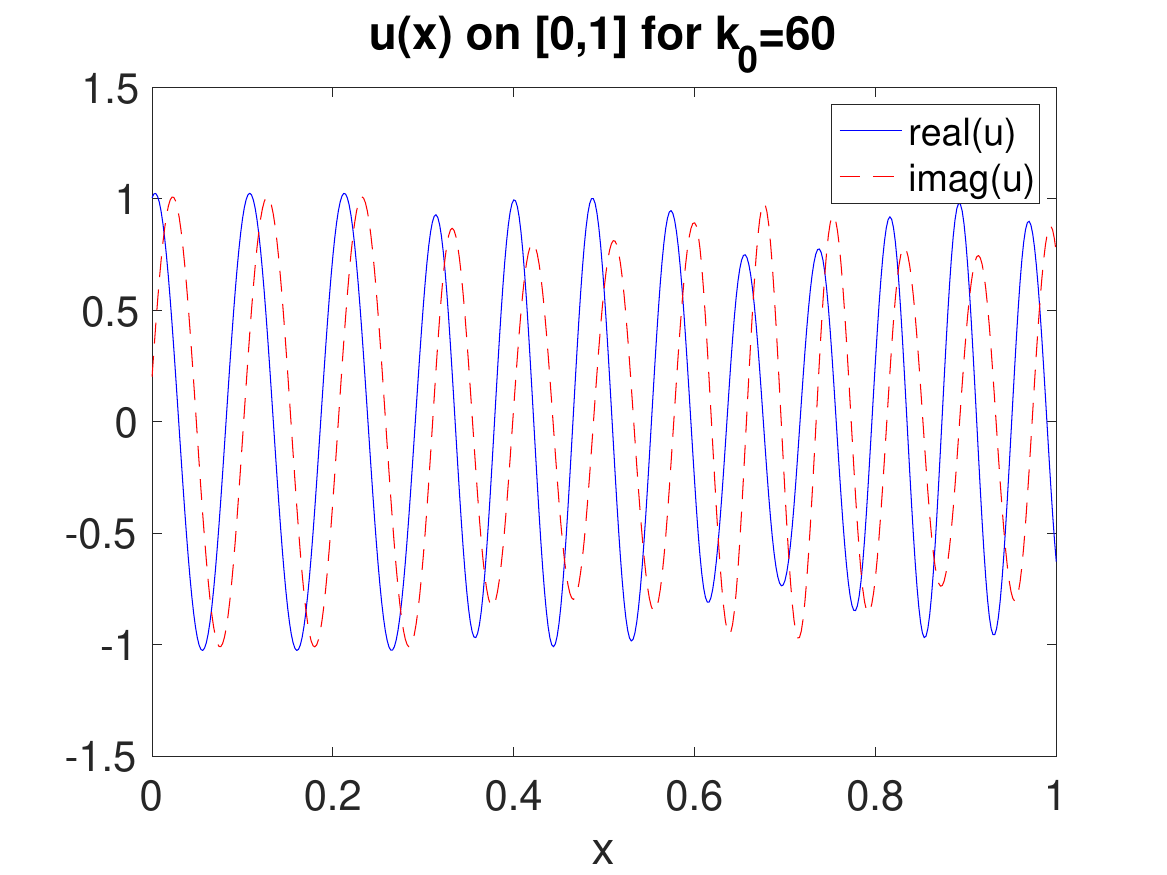}
\caption{\label{NLH1}  Shown above are solution plots for NLH with varying linear wave number $k_0$.}
\end{center}
\end{figure}

Even in 1D, NLH is a very challenging problem, especially for larger values of $k_0$ since this increases the strength of the cubic nonlinearity \cite{BFT07,BFT09,FT01}.  NLH has recently become a benchmark for testing of nonlinear solvers and acceleration methods, e.g. \cite{PR25,YTA22,PR21}.  

We consider the Picard iteration for solving NLH, which is given by
\begin{align}
\frac{d^2 u_{j+1}}{dx^2} + k_0^2 \left( 1 + \epsilon(x)  |u_j|^2 \right) u_{j+1} & = 0, \ \ \ 0<x<1, \label{nlh1} \\
\frac{du_{j+1}}{dx} + i k_0 u_{j+1} & = 2 i k_0,\ \ \ x=0, \label{nlh2} \\
\frac{du_{j+1}}{dx} - i k_0 u_{j+1} & = 0,\ \ \ x=1. \label{nlh3}
\end{align}
The Picard iteration \eqref{nlh1}-\eqref{nlh3} is a fixed point iteration, since we can write $u_{j+1}=q(u_j)$ where 
$q$ is the solution operator of \eqref{nlh1}-\eqref{nlh3}.   
Following \cite{BFT07}, we take $u_0=e^{ik_0 x}$.  Note that due to the second derivative in NLH (and the Picard iteration), the natural function space is $H = H^1(\Omega)$ with two way boundary conditions.  The Picard system is discretized using second order finite differences (which is equivalent to using $P_1$ finite elements in this setting) and uniform point spacing of $h=0.002$ using to the iteration.  We approximate the discrete dual norm for the NGMRES optimization problem using 
\[
\| \phi \|_{H'} \approx (\phi, S^{-1} \phi)^{1/2},
\]
where $S$ is the finite difference stiffness matrix with endpoint degrees of freedom removed.  We note with this discretization and problem setup, the Picard iteration converges for $k_0 \le 17$, and diverges for $k_0>18$ (tests omitted).  Linear solves are performed with a direct solver for this problem.

\subsubsection{Illustration of NGMRES convergence theory}

For our first test of NGMRES applied to NLH and its associated Picard iteration, we consider $k_0=20$ and run NGMRES with $m$=2, 5, 10 and 20, as well as the Picard iteration with no acceleration.  Results are shown in Figure \ref{NLH2} for convergence, $\gamma_k$ and $\theta_k$.  For convergence, we observe that unaccelerated Picard fails to converge, while all NGMRES-Picard tests converge.  Convergence improves as $m$ is increased, except $m=10$ and $m=20$ perform about the same.  Acceleration coefficients are shown in Figure \ref{NLH2} at the center, and as expected we see that overall, the acceleration coefficients get smaller as $m$ gets larger.  Shown on the right in Figure \ref{NLH2} are plots of the actual convergence rate $\frac{\| g(u_{k}) \|_{H'}}{ \| g(u_{k-1}) \|_{H'}}$ and predicted linear convergence rate $\theta_k$ for $m=5$.  We observe that after about the tenth iteration, $\theta_k$ is a very good predictor of the actual convergence rate; once the nonlinear residual is sufficiently small, the higher order terms in the convergence estimate become negligible and $\theta_k$ becomes a very good predictor of the convergence rate at iteration $k$.

\begin{figure}[ht!]
\begin{center}
\includegraphics[width = .32\textwidth, height=.32\textwidth,viewport=0 0 530 430, clip]{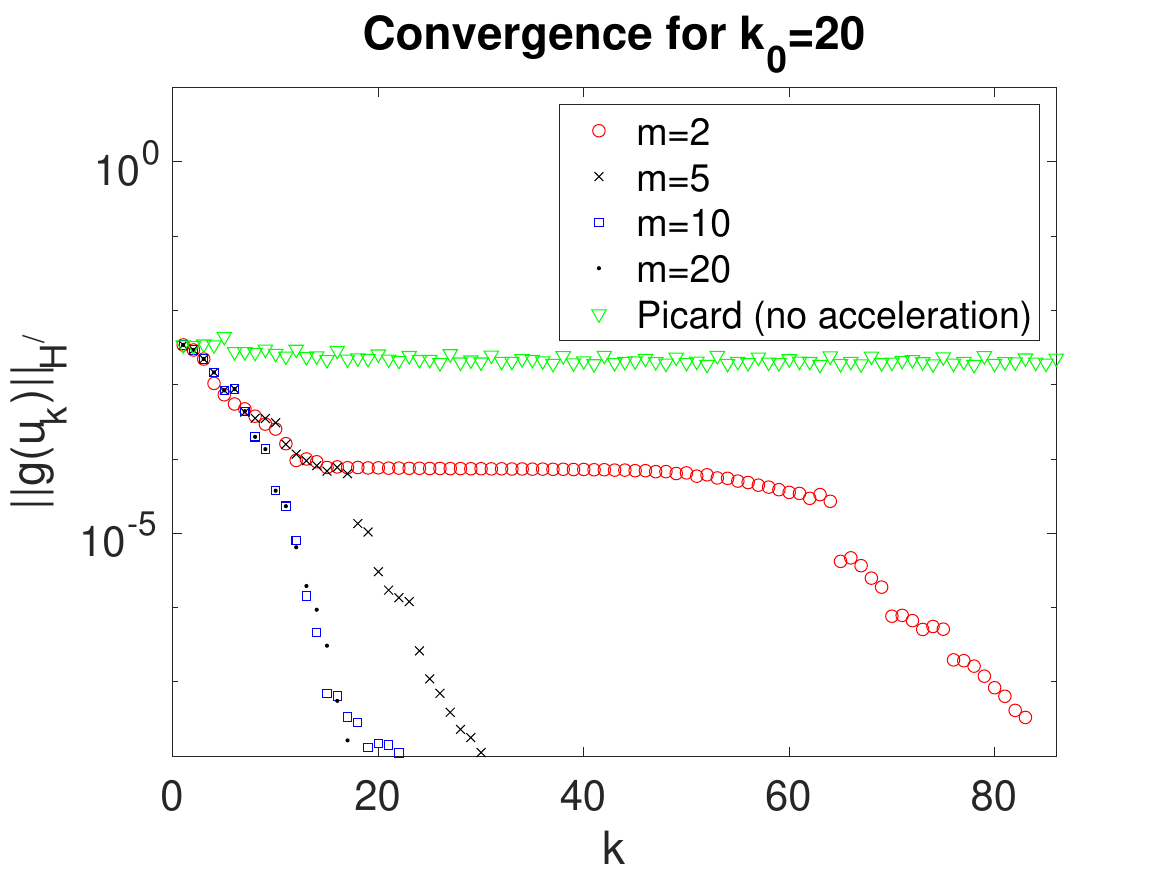}
\includegraphics[width = .32\textwidth, height=.32\textwidth,viewport=0 0 530 430, clip]{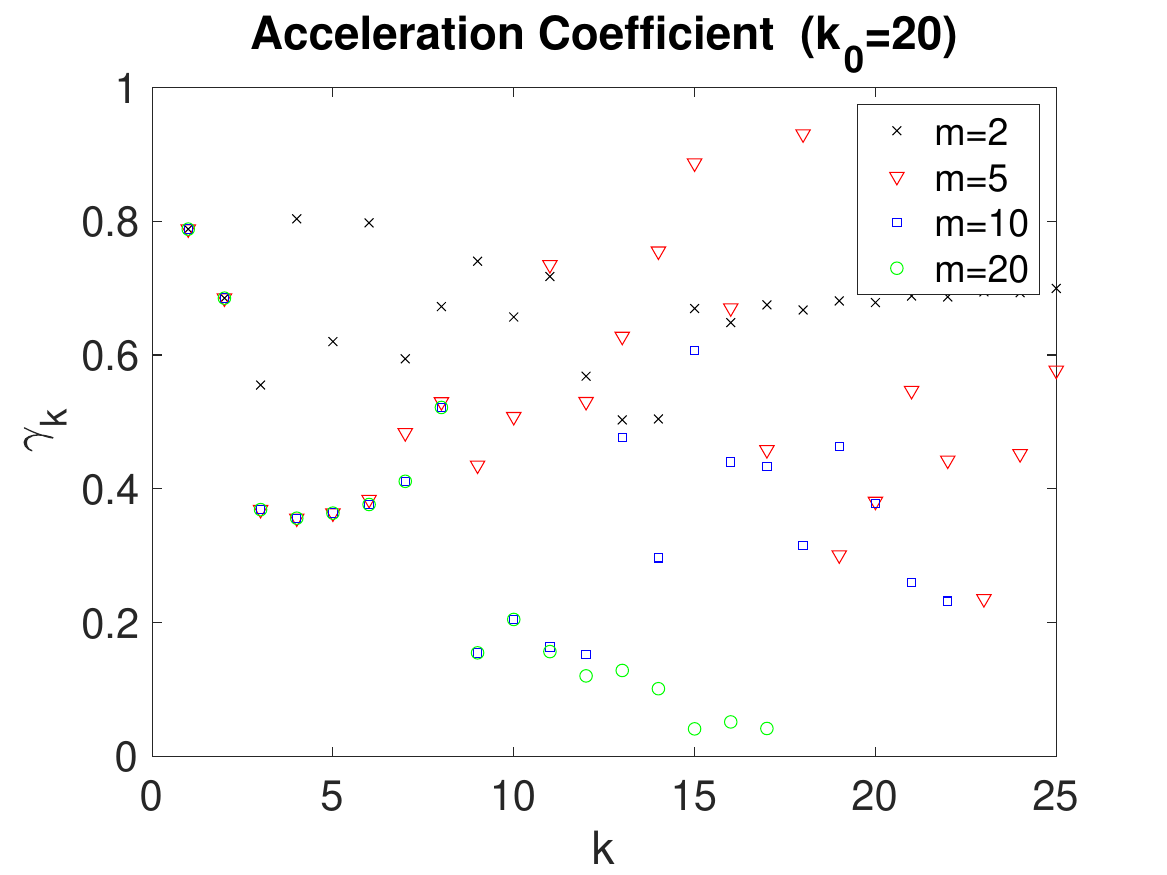}
\includegraphics[width = .32\textwidth, height=.32\textwidth,viewport=0 0 530 430, clip]{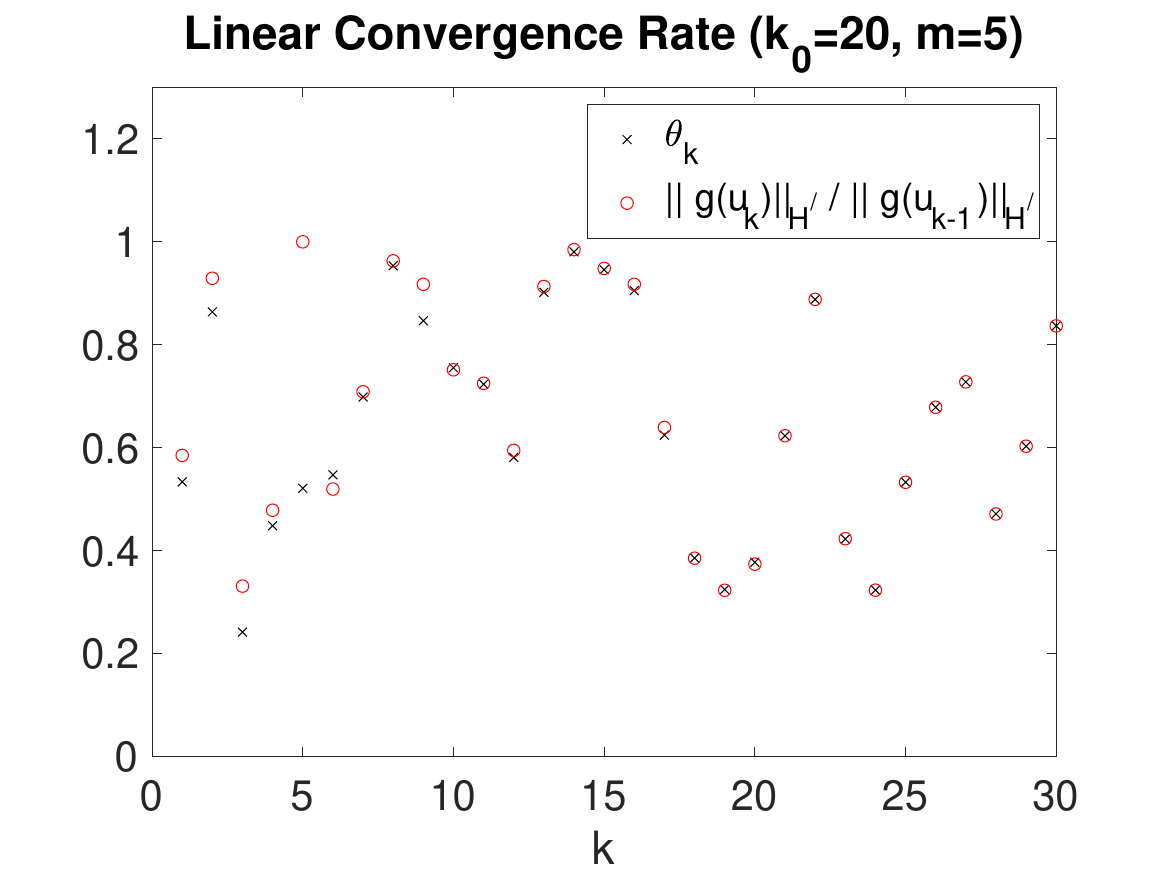}
\caption{\label{NLH2}  Shown above are plots of NGMRES applied to NLH-Picard with $k_0=20$ and varying $m$: convergence (left), acceleration coefficients $\gamma_k$ (center), and a comparison of the actual convergence rate $\| g(u_{k+1}) \|_{H'} / \| g(u_k) \|_{H'}$ with the predicted linear convergence rate $\theta_k$ (right).  }
\end{center}
\end{figure}

\subsubsection{Improving NGMRES with restarts}

\begin{figure}[ht!]
\begin{center}
\includegraphics[width = .32\textwidth, height=.32\textwidth,viewport=0 0 530 430, clip]{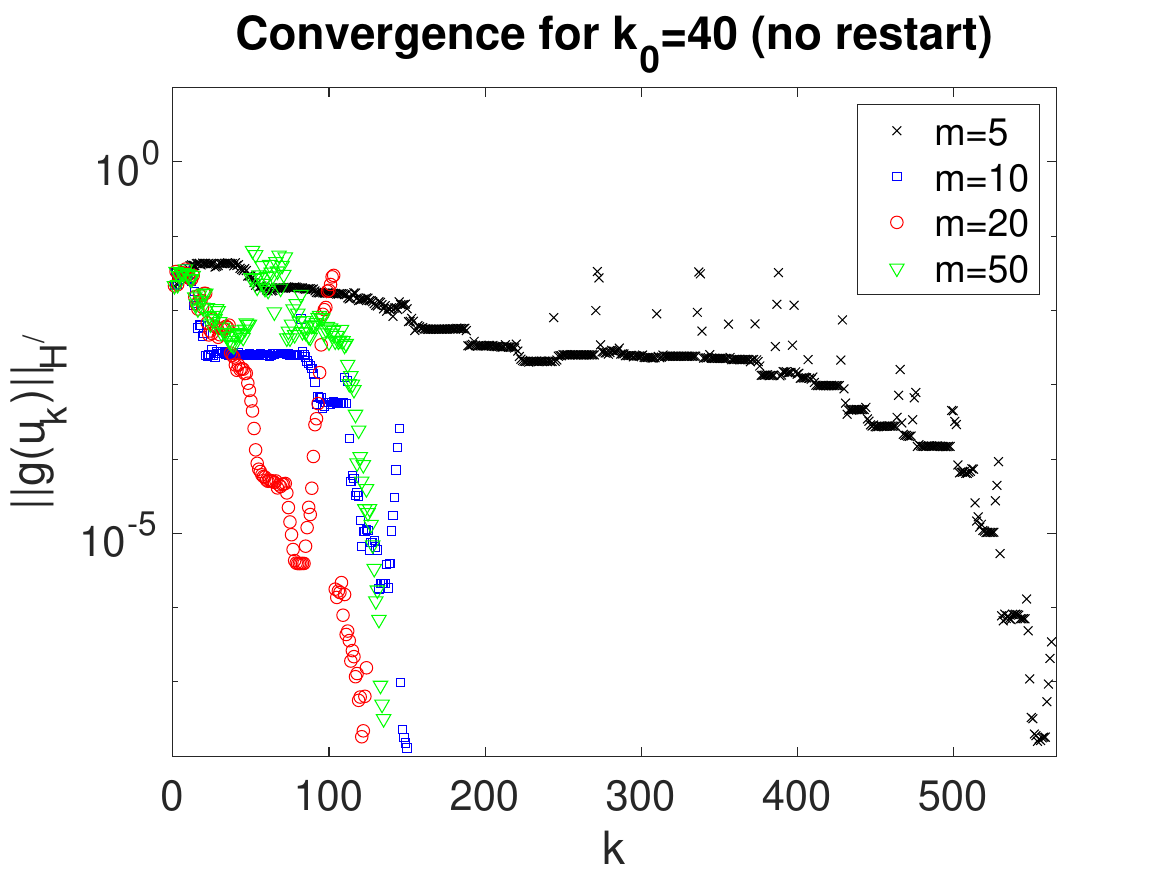}
\includegraphics[width = .32\textwidth, height=.32\textwidth,viewport=0 0 530 430, clip]{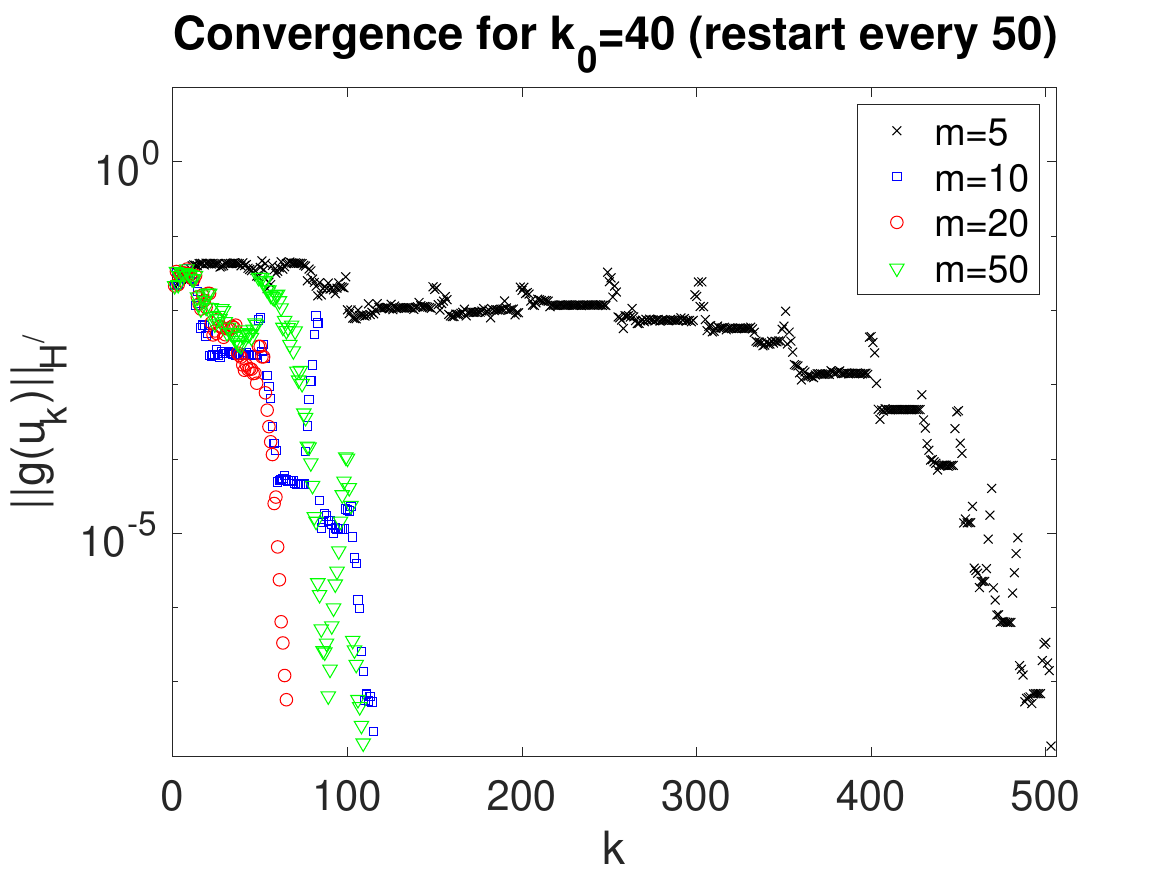}
\includegraphics[width = .32\textwidth, height=.32\textwidth,viewport=0 0 530 430, clip]{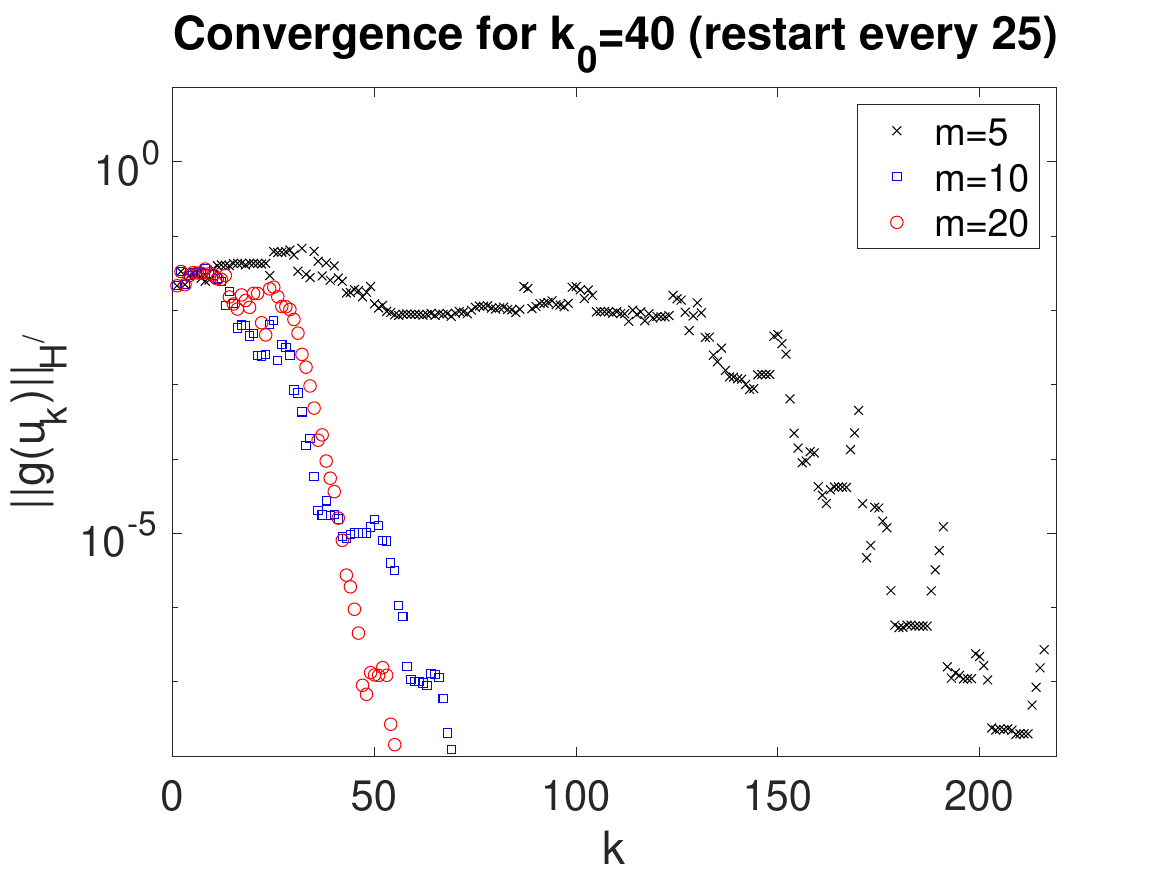}
\caption{\label{NLH3}  Shown above are plots of NGMRES applied to NLH-Picard with $k_0=40$ and varying $m$, with no restarts (left), restarts every 50 iterations (center) and restarts every 25 iterations (right).}
\end{center}
\end{figure}

Theorem \ref{thm:ngmres2} shows that in the noncontractive case, the nonlinear residual bound higher order term sum goes back to the initial iterate.  If sharp, this suggests that gains made on the linear terms can be dominated by the higher order terms as more iterations are performed.  Hence restarts can be a way to mitigate this effect.  To test this, we pick $k_0=40$ and repeat the same NLH test as above with varying $m$.  Results are shown in Figure \ref{NLH3}, and we see that while each of $m$=5, 10, 20 and 50 converge (note $m<5$ failed; plots omitted), convergence is quite erratic when no restarts are used.  Using restarts every 50 iterations improves convergence for each choice of $m$, and also yields somewhat smoother convergence behavior.  Even more improvement is found when restarts are used every 25 iterations.

\subsubsection{Dual norm versus $\ell^2$ for the optimization norm}

Our analysis requires the use of the dual norm $H'$ for the optimization problem, however historically NGMRES users choose the $\ell^2$ norm for convenience and simplicity (but in their defense, until now there was no general convergence theory that suggested using a different norm than $\ell^2$).  

We test NGMRES-Picard for NLH using the $H'$ and $\ell^2$ optimization norms, under three parameter sets: $[k_0$=40, $m$=10, restart=25]; $[k_0$=60, $m$=10, restart=25]; $[k_0$=60, $m$=50, restart=none].  Results are shown in Figure \ref{NLH4}, and we observe that for the easier case of $k_0=40$, results are similar for the $H'$ and $\ell^2$ optimization norm choices.  However, for the harder case of $k_0$=60, the results using the $H'$ norm are much better than those using $\ell^2$; for $m$=10 and restart=25, convergence is achieved in 475 iterations using $H'$ optimization norm while convergence is never reached using $\ell^2$ (in 2000 iterations, plot only shows up to 700 for comparison purposes).  For $m$=50 and no restarts, using optimization norm $H'$ provides convergence in about 450 iteration while it takes 625 for convergence using $\ell^2$.

These results are similar to what was found in \cite{HR26} for NGMRES applied to the Picard iteration for NSE.  That, for easier problems, using the dual norm for the optimization problem or using $\ell^2$ gives similar convergence results.  However for harder problems, especially in 3D, using the dual norm gives much better convergence results.

\begin{figure}[ht!]
\begin{center}
\includegraphics[width = .6\textwidth, height=.4\textwidth,viewport=0 0 900 530, clip]{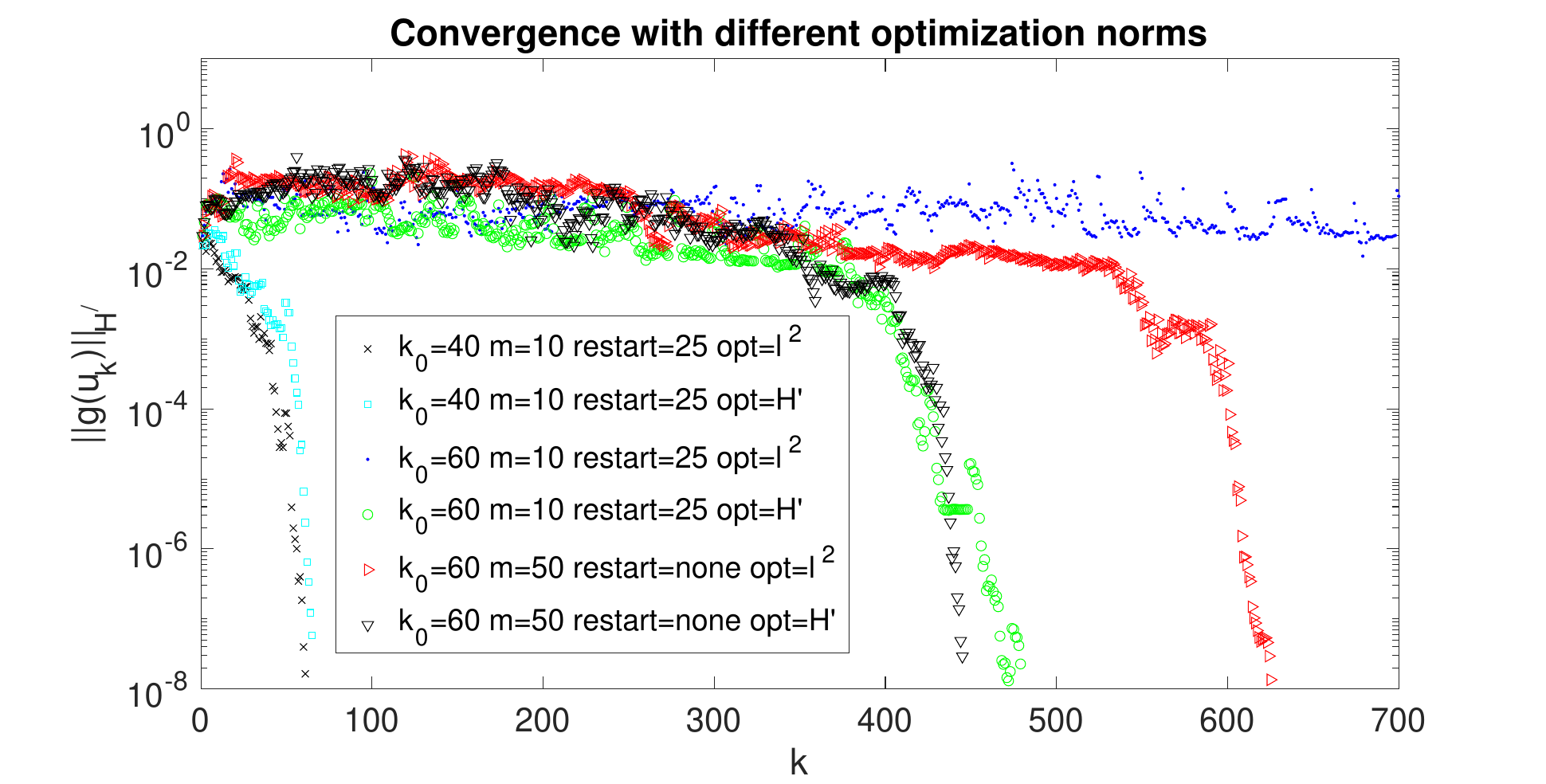}
\caption{\label{NLH4}  Shown above are plots of NMGRES applied to NLH-Picard with varying $k_0$, $m$ and restarts, using the dual norm and the $\ell^2$ norm for the NGMRES optimization problem.}
\end{center}
\end{figure}

\subsection{3D Buoyancy driven flow (Boussinesq equations)}

We next consider tests on flows driven by natural convection or buoyancy.   These flows occur in applications such as ventilation, solar collectors, window insulation, cooling in electronics,
and many others \cite{CK11}, and are modeled by
the Boussinesq system: in $\Omega\subset\mathbb{R}^d$ ($d$=2 or 3) by
\begin{align}\label{eqn:bouss}
u_t + (u\cdot \nabla ) u - \nu \Delta u + \nabla p 
&= Ri \left( \begin{array}{c} 0 \\ 0 \\ \theta \end{array} \right) + f,
\nonumber \\
\nabla \cdot u &= 0,
\nonumber \\
T_t + (u\cdot \nabla) T - \kappa \Delta T &= F,
\end{align}
with $u$ representing the velocity field, $T$ the temperature 
(or density), $p$ the pressure, and $f$ and $F$ the external forces.  The kinematic viscosity $\nu>0$ is defined to be the inverse of the Reynolds number ($Re=\nu^{-1}$), and the thermal conductivity $\kappa$ is given by
$\kappa=Re^{-1}Pr^{-1}$,  with $Pr$ representing the Prandtl number and $Ri$ the Richardson number accounting for the gravitational force.
The associated Rayleigh number is defined by 
\[
Ra=Ri\cdot Re^2 \cdot Pr.
\] 
Larger $Ra$ leads to more complex physical behavior as well as more difficulties in numerically solving the equations.

\begin{figure}[ht!]
\begin{center}
\includegraphics[width = .25\textwidth, height=.25\textwidth,viewport=0 0 530 420, clip]{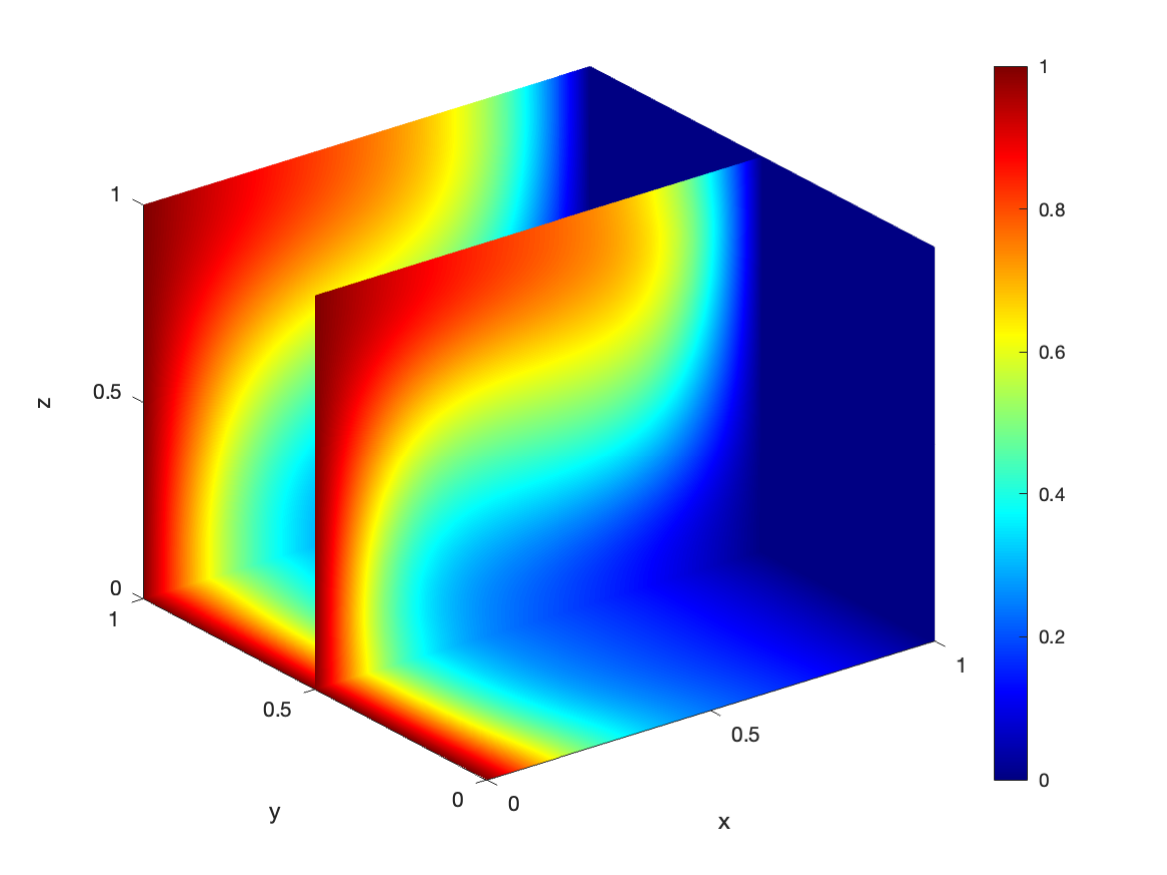}
\includegraphics[width = .73\textwidth, height=.25\textwidth,viewport=100 50 1070 350, clip]{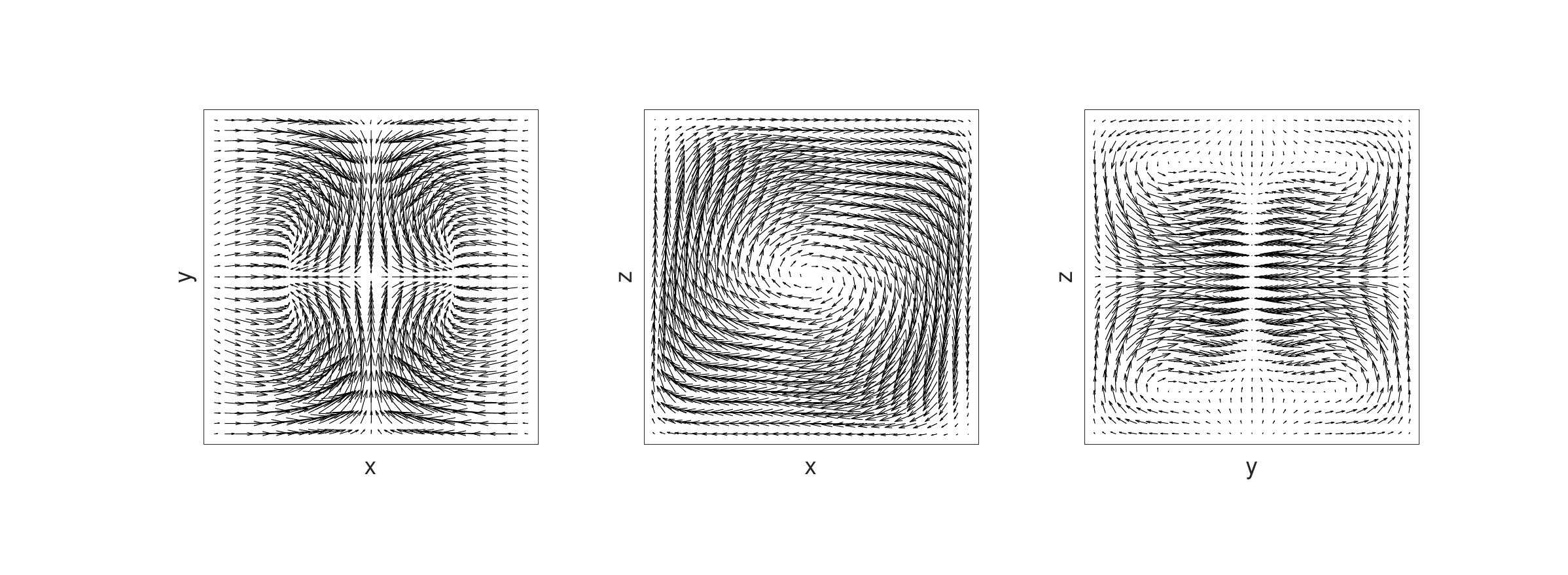}\\
\includegraphics[width = .25\textwidth, height=.25\textwidth,viewport=0 0 530 420, clip]{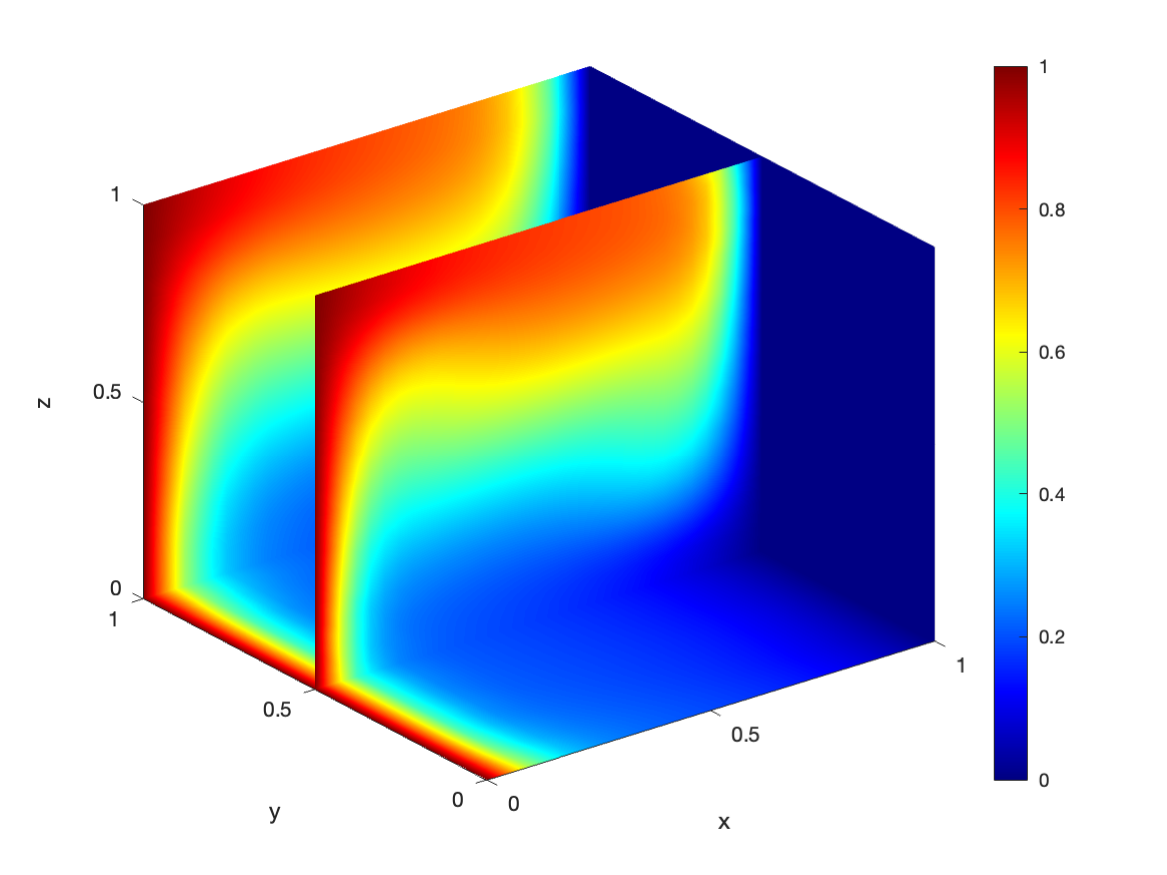}
\includegraphics[width = .73\textwidth, height=.25\textwidth,viewport=100 50 1070 350, clip]{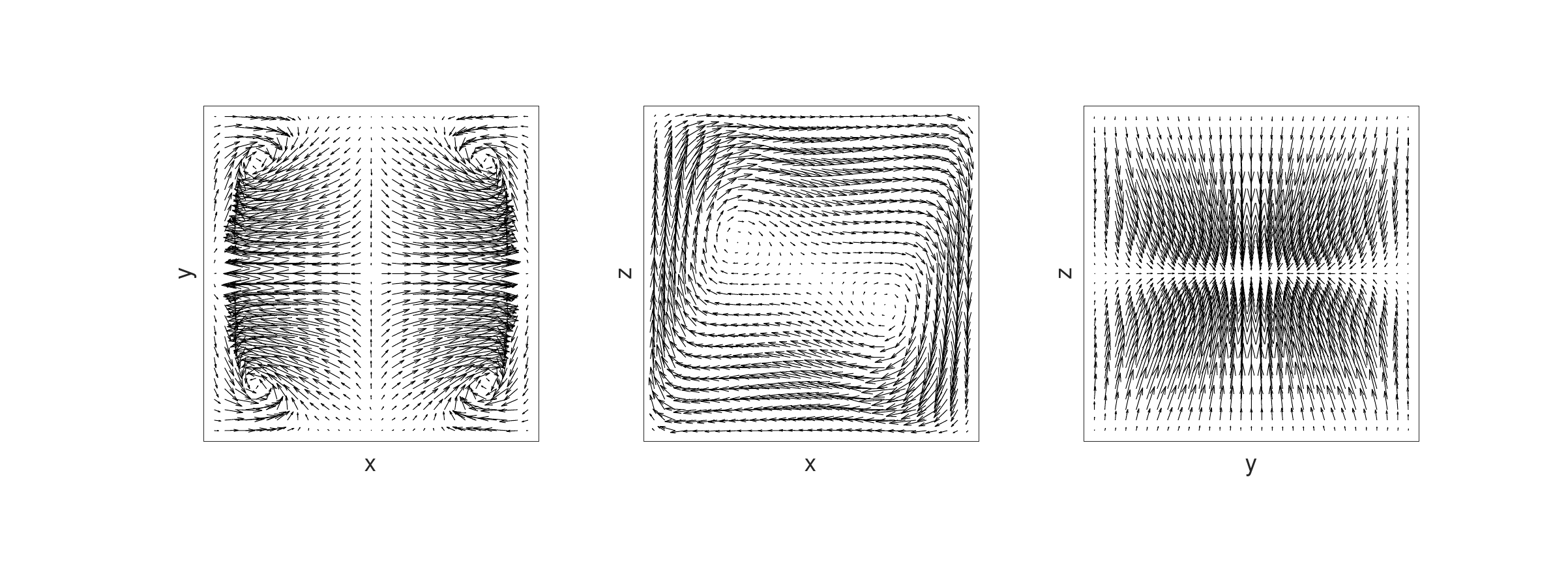}
\caption{\label{B0}  Shown above are solution plots of the Boussinesq 3D differentially heated cavity test problem with $Pr=0.71$ as temperature (left) and midsliceplane velocity fields (right), with $Ra=10000$ (top) and $Ra=100000$ (bottom).}
\end{center}
\end{figure}

For a test example, we use the 3D differentially heated cavity benchmark problem, where $\Omega=(0,1)^3$, there is no external forcing ($f=0$, $F=0$), no slip velocity boundary conditions are enforced on all sides, and for temperature we strongly enforce $T(0,y,z)=1$, $T(1,y,z)=0$ and perfect insulation $(\nabla T\cdot n=0)$ is weakly enforced on the other four sides.  We set $Ri=1$, $Pr=0.71$ to model air, and then vary $Re$ in order to vary $Ra$.

We solve directly for steady solutions (i.e. $u_t=0$ and $T_t=0$) using the Picard iteration:  Given $u_0$, define $u_{k},T_{k}$, $k\ge 1$, to be solutions of the (decoupled) system
\begin{align}
(u_{k-1}\cdot \nabla ) u_k - \nu \Delta u_k + \nabla p_k & = Ri [ 0;T_{k} ], \label{it1a} \\
\nabla \cdot u_k &= 0, \label{it2a} \\
(u_{k-1}\cdot \nabla) T_k - \kappa \Delta T_k & = 0, \label{it3a}
\end{align}
and the solutions also satisfy the boundary conditions discussed above.  

Note that \eqref{it3a} can be decoupled and solved to obtain $T_k$, and then an Oseen solve of \eqref{it1a}-\eqref{it2a} can be performed for $u_k$ and $p_k$.  We note that for Boussinesq, one iteration of Picard is much cheaper than one Newton iteration, which would include the terms $+(u_k-u_{k-1})\cdot\nabla u_{k-1}$ on the left hand side of the momentum equation and $+(u_k-u_{k-1})\cdot\nabla T_{k-1}$ on the left hand side of the energy equation.  Hence, the Newton linear systems would be fully coupled and much more difficult to efficiently and reliably solve.  Thus, accelerating the Picard iteration for Boussinesq is highly desirable.

We discretize velocity-pressure using $(P_3,P_2^{disc})$ Scott-Vogelius (SV) elements on a tetrahedral grid constructed with a rectangular box grid generated by $12^3$ Chebyshev points that is then refined by splitting each box into 6 tetrahedra and finally barycenter refining each of those tetrahedra.  This SV pair is known to be inf-sup stable on such a grid \cite{Z05}.  For the temperature, we use $P_3$ on this same mesh, and in total the discretization yields 954K degrees of freedom (dof).  The temperature linear solves are done directly, and for the Oseen solves we use an augmented Lagrangian type preconditioner for GMRES following \cite{benzi,HR13,BL12}, with direct inner solves.  

The natural function spaces of velocity and temperature are (after equivalently switching to homogeneous Dirichlet boundary conditions for temperature on the walls and adjusting the temperature equation forcing appropriately),
\[ 
V = \{ v\in H^1_0(\Omega)^3, \ \nabla \cdot v=0\},\ \ \
X = \{ v\in H^1(\Omega), v(x,y,z)=0 \mbox{ for } x=0,1\}.
\]
We denote $q_B: V\times X \rightarrow V\times X$ as the solution operator of \eqref{it1a}-\eqref{it3a} and thus also the fixed point function: $[u_{k+1},T_{k+1}]=q_B([u_k,T_k])$.  Note the pressure and pressure space is removed since the SV pair is inf-sup stable and thus the divergence-free velocity subspace can be used.  We denote $g_B: V\times X \rightarrow V'\times X'$ to be the nonlinear residual operator.

We approximate the velocity and temperature dual space norms for the NGMRES optimization using
\[
\| \phi \|_{V'} \approx (\phi,A_h^{-1} \phi)^{1/2},\ \ \| \chi \|_{X'} \approx (\chi,S_h^{-1} \chi)^{1/2},
\]
where $A_h$ is the discrete Stokes operator and $S_h$ is the stiffness matrix, each with its associated Dirichlet boundary dof removed.  Hence the optimization dual norm used is
\[
\| (\phi,\chi) \|_{B'} := \left( \| \nabla \phi \|_{V'}^2 + \| \nabla \chi \|_{X'}^2 \right)^{1/2}.
\]
This norm is also used for measuring convergence in our tests.  We note that the Picard iteration converges for $Ra$=10000 in 85 iterations, and for $Ra\ge 25000$ it does not converge in 200 iterations (and shows no indication that it will ever converge).  

\subsubsection{Illustration of NGMRES convergence theory}

\begin{figure}[ht!]
\begin{center}
\includegraphics[width = .32\textwidth, height=.32\textwidth,viewport=0 0 530 430, clip]{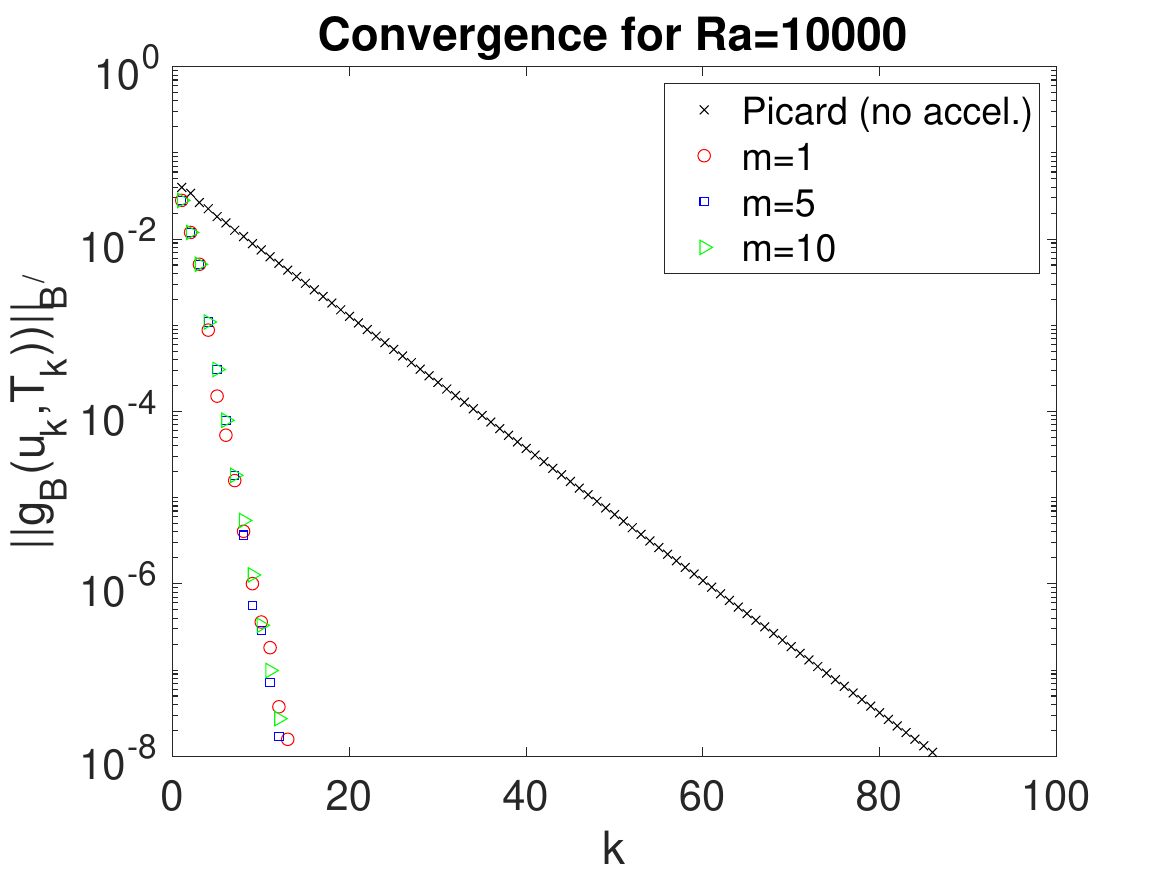}
\includegraphics[width = .32\textwidth, height=.32\textwidth,viewport=0 0 530 430, clip]{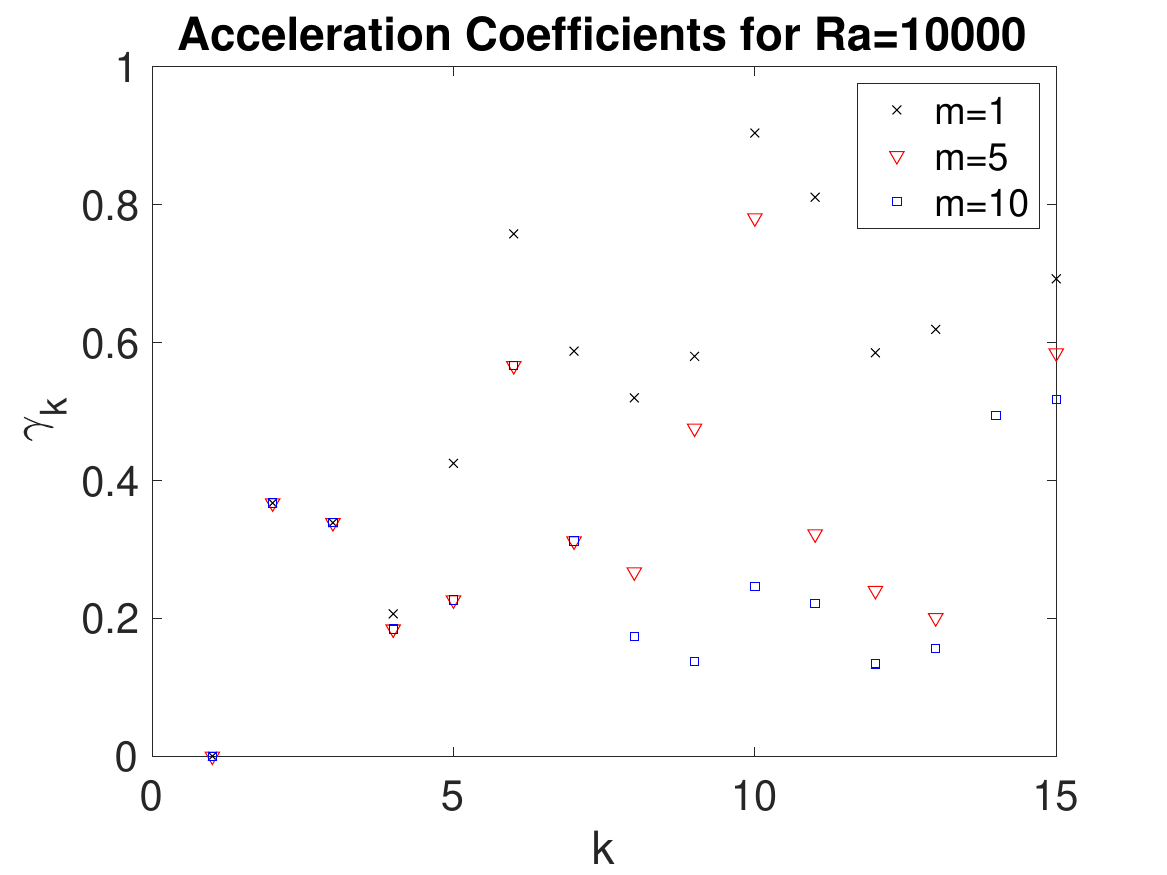}
\includegraphics[width = .32\textwidth, height=.32\textwidth,viewport=0 0 550 430, clip]{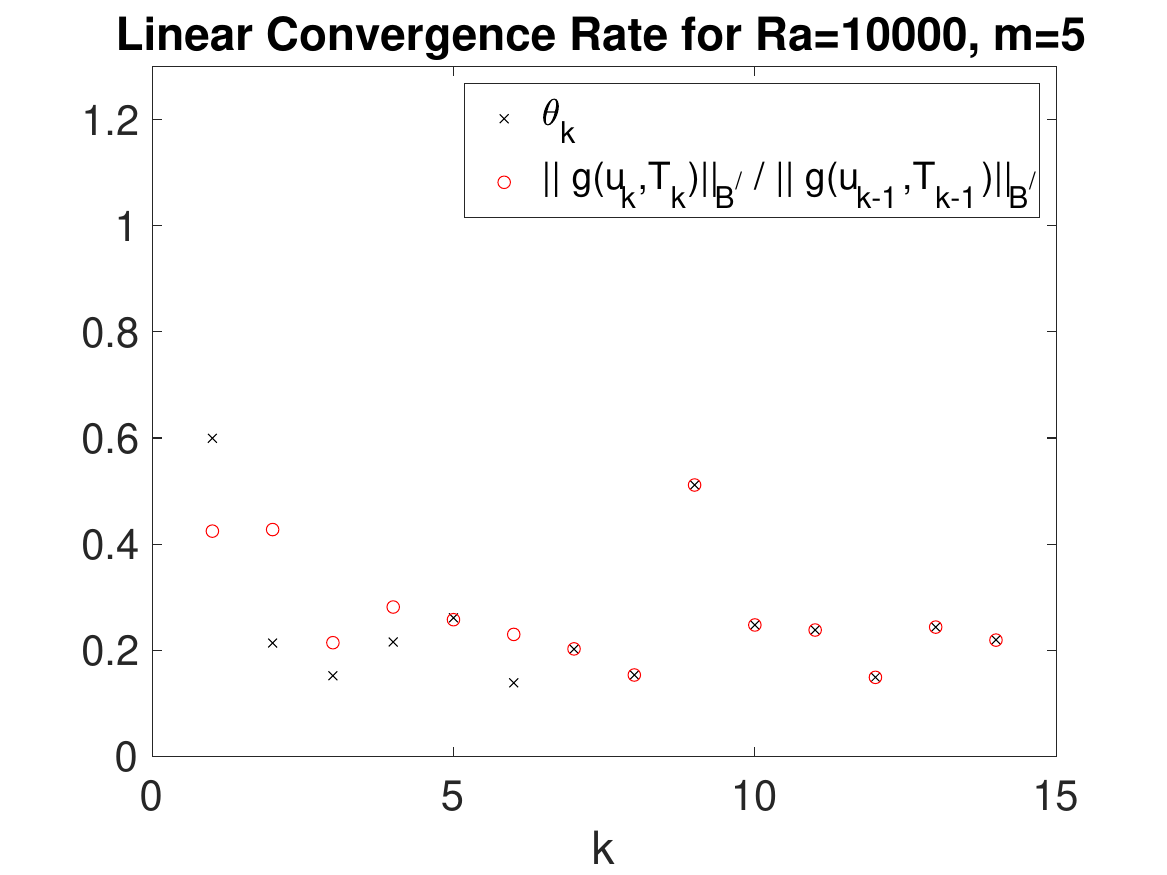}\\
\includegraphics[width = .32\textwidth, height=.32\textwidth,viewport=0 0 530 430, clip]{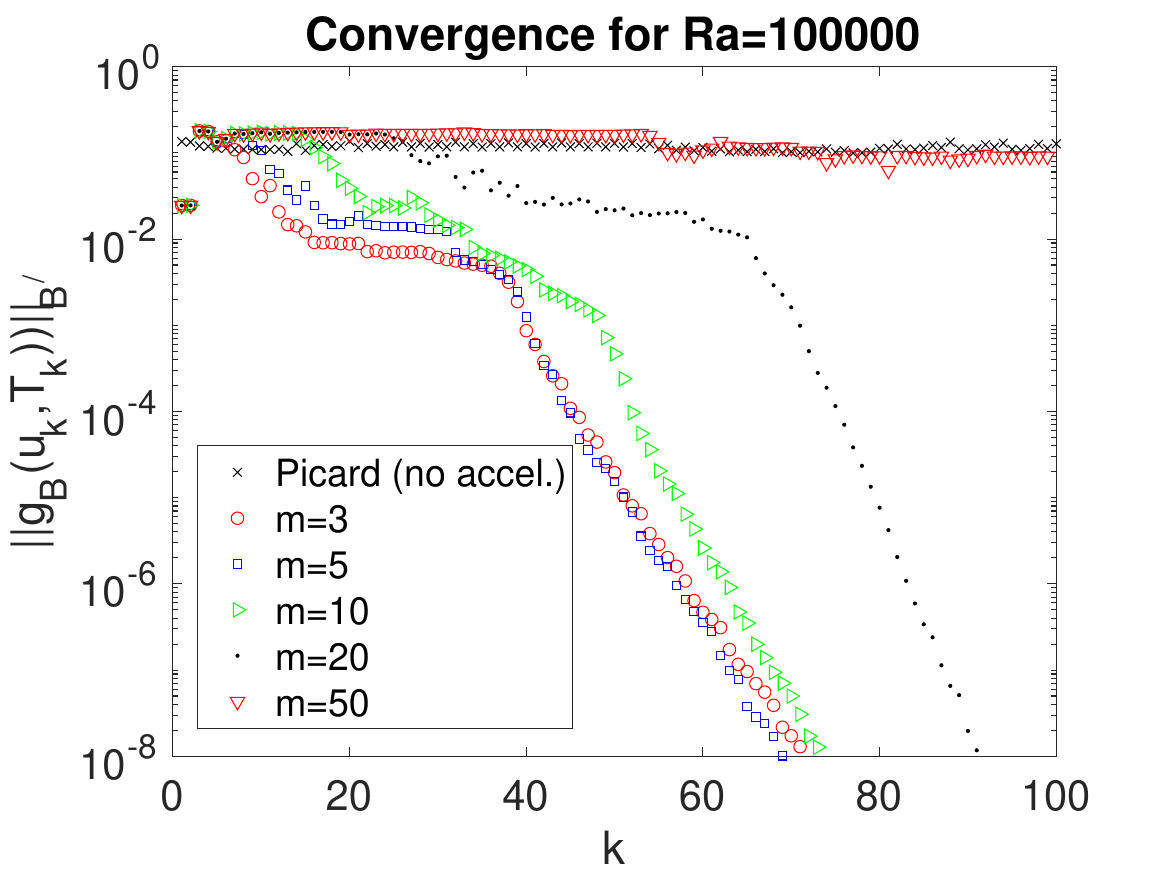}
\includegraphics[width = .32\textwidth, height=.32\textwidth,viewport=0 0 530 430, clip]{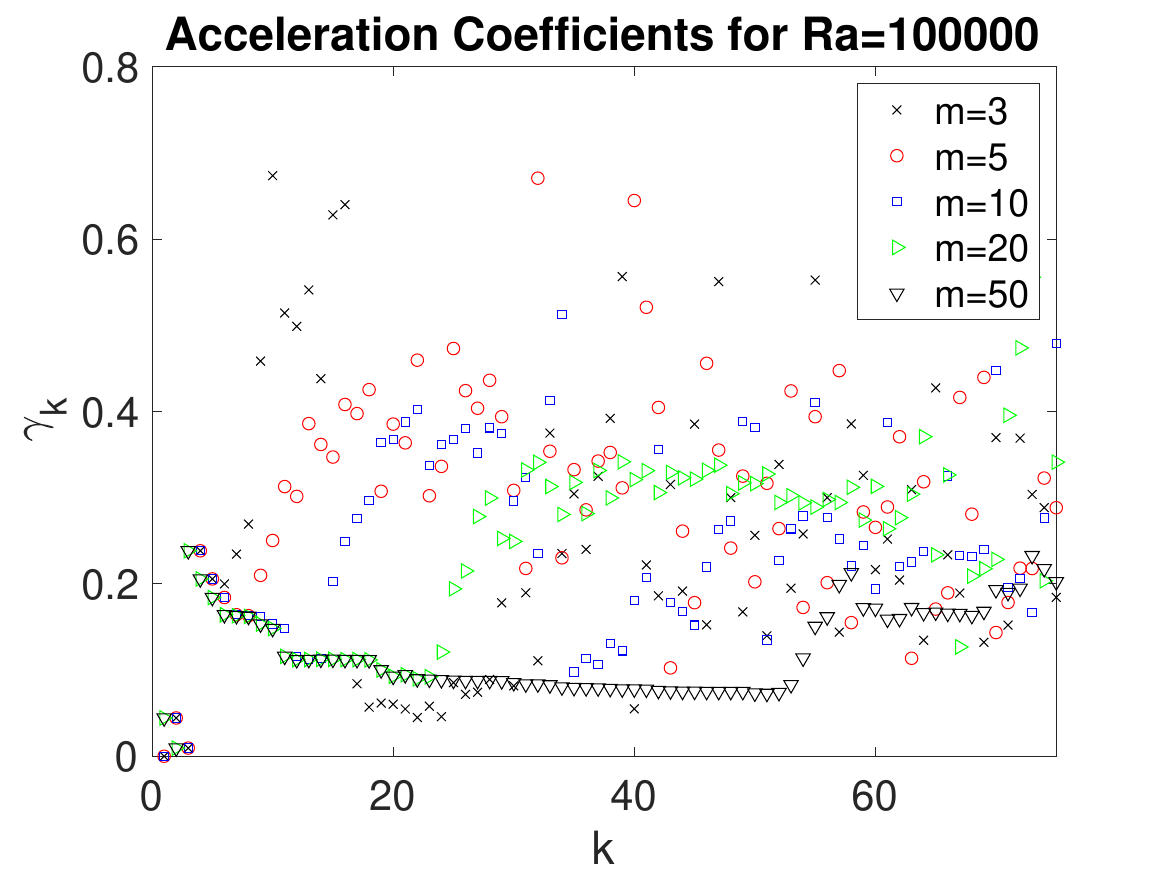}
\includegraphics[width = .32\textwidth, height=.32\textwidth,viewport=0 0 550 430, clip]{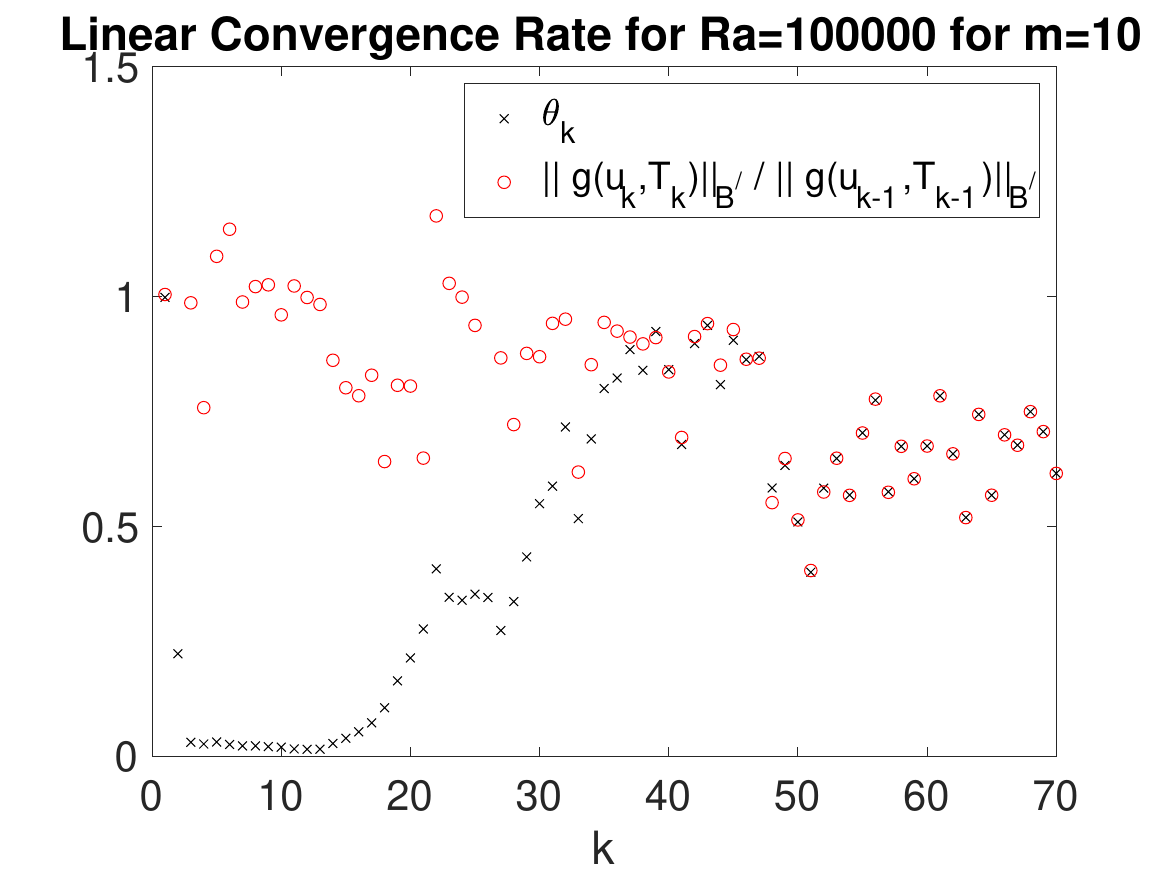}
\caption{\label{B2}  Shown above are results for  NGMRES-Picard for Boussinesq with $Ra=10000$ (top) and $100000$ (bottom) and varying $m$: convergence (left), acceleration coefficients $\gamma_k$ (center), and a comparison of the actual convergence rate $\frac{ \| g(u_{k},T_{k}) \|_{B'}}{ \| g(u_{k-1},T_{k-1}) \|_{B'}}$ with the predicted linear convergence rate $\theta_k$ (right). }
\end{center}
\end{figure}

For our first test with the Boussinesq system, we compute NGMRES-Picard with varying $m$ and $Ra$=10000 and 100000.  Results are shown in Figure \ref{B2}.  For $Ra$=10000, we observe that the unaccelerated Picard iteration converges in about 85 iterations, while NGMRES-Picard converges in about 15 iterations for each of $m=1,5,10$.  Acceleration coefficients $\gamma_k$ for $Ra$=10000 are as expected in that we see an overall decrease as $m$ increases.  We also observe that $\theta_k$ becomes an excellent predictor for the convergence rate at each step, after about the 7th iteration.

For $Ra$=100000, unaccelerated Picard fails to converge.  However, NGMRES-Picard converges for each of $m=3,5,10$, and their respective convergence behaviors are similar.  Worse converge is observed for  $m=20$, with convergence in 90 iterations, while for $m=50$ there is no convergence at all.  The plot of $\gamma_k$'s shows an overall decrease as $m$ increases, but as our theory shows, this only affects the linear convergence rate and for large $m$ the higher order terms can play a significant role in slowing/preventing convergence.  We also observe that for $m=10$, $\theta_k$ agrees well with the actual convergence ratio by about iteration 50.  

\subsubsection{Using $\theta_k$ for an adaptive depth strategy}

\begin{figure}[ht!]
\begin{center}
\includegraphics[width = .32\textwidth, height=.32\textwidth,viewport=0 0 530 430, clip]{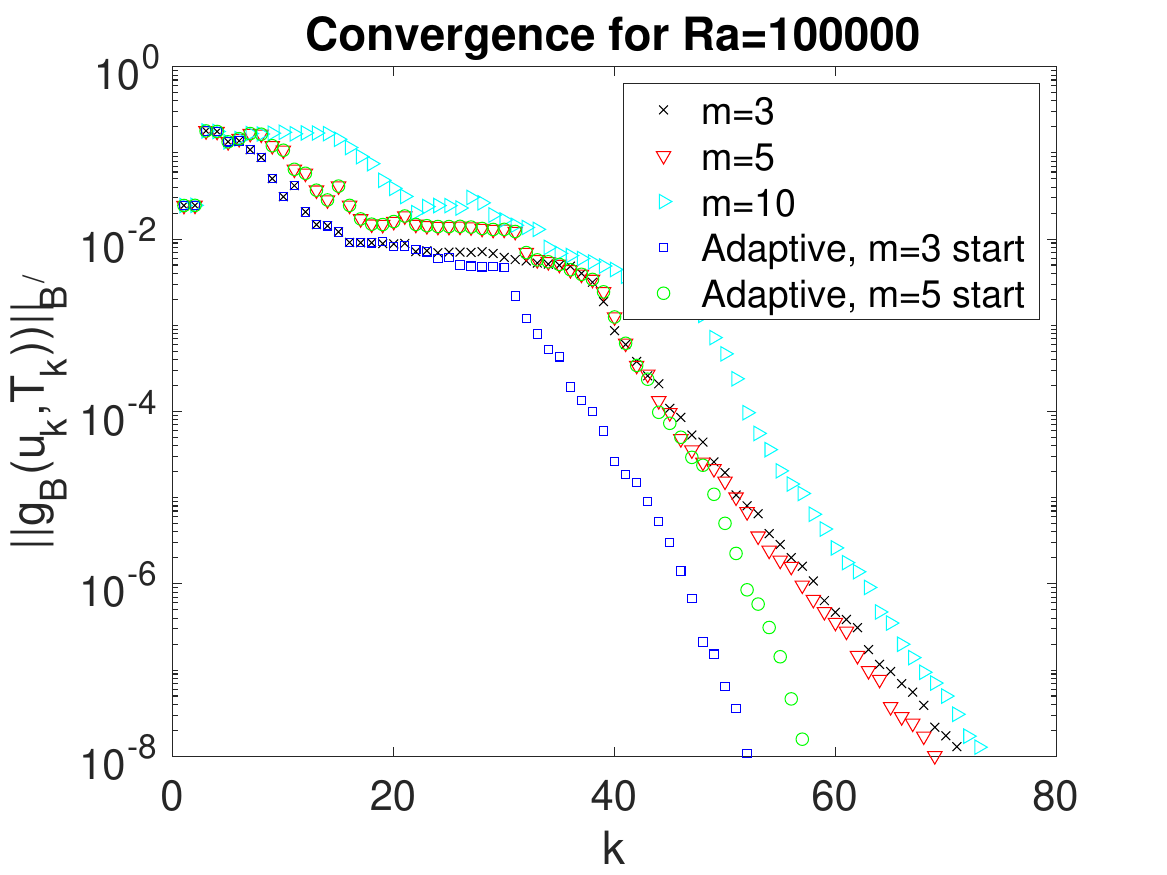}
\includegraphics[width = .32\textwidth, height=.32\textwidth,viewport=0 0 530 430, clip]{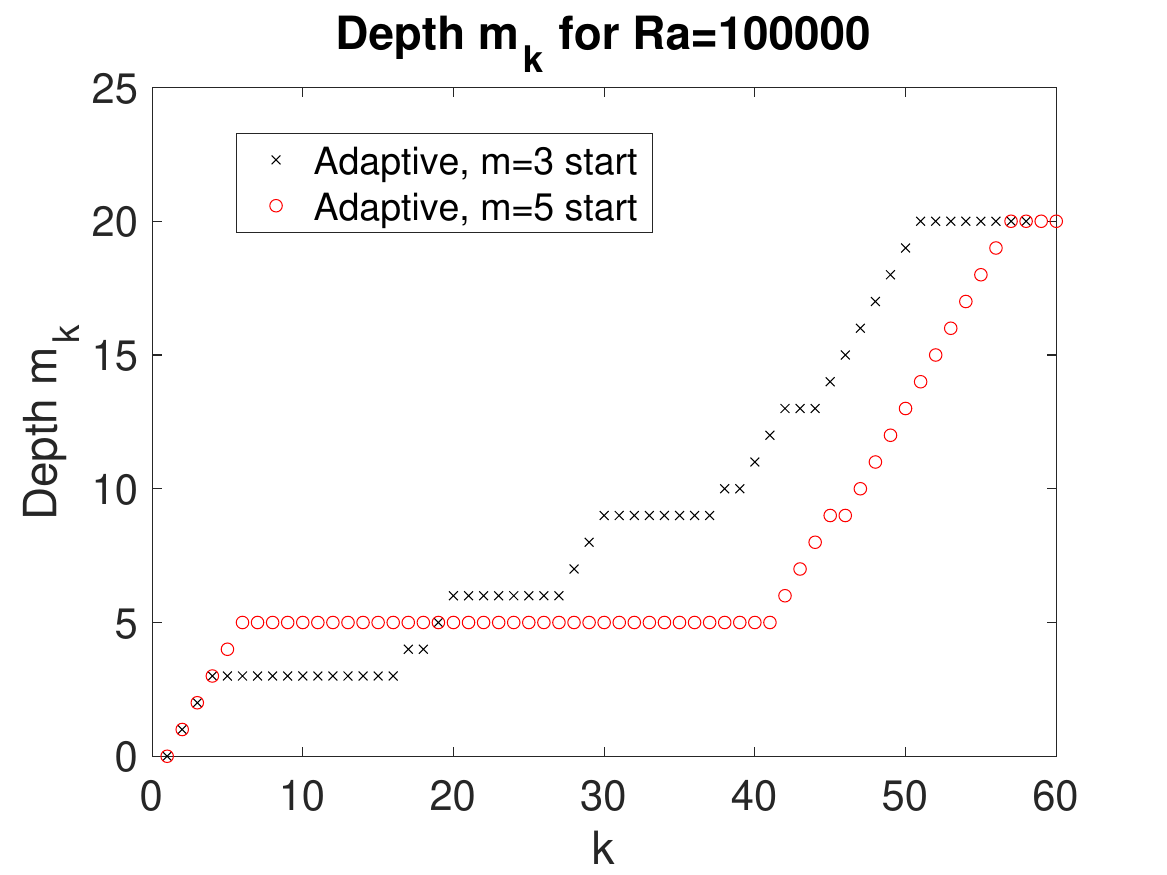}
\caption{\label{B3}  Shown above are results for  NMGRES-Picard for Boussinesq with $Ra=100000$ using constant $m$ and adaptive $m$ (left), and $m_k$ vs $k$ (right). {Here, $m_k= \min \{ m_{max},k \}$.}}
\end{center}
\end{figure}

As observed in our theory and in the numerical tests above, it is often best to use smaller $m$ early in the iteration when the nonlinear residual is large and then larger $m$ once the residual is small.  The key idea is that increasing $m$ improves the linear convergence rate, but at the expense of higher order terms having a greater effect.  Hence once the nonlinear residuals are moderately small, then the higher order terms become negligible and so increasing $m$ will only help.

Based on our theory above in the contractive regime, we can use $\theta_k$'s sharp prediction of the 
linear convergence rate to better choose when to increase $m$.  Following along the lines we proposed  in \cite{HR26b}, we increase $m$ in the following way.  First, a small $m$ is chosen and NGMRES is run as usual.  At each step, we check the inequality
\[
\bigg| \theta_k - \frac{ \| g_B(u_k,T_k)\|_{B'}}{ \| g_B(u_{k-1},T_{k-1})\|_{B'}}\bigg| <0.001.
\]
If this is satisfied, it strongly suggests that the higher order terms are not affecting the convergence rate, and so we increase the maximum $m$ by one at the next iteration.  

Convergence results for this strategy are shown in Figure \ref{B3} for $Ra$=100000.  We observe that by starting with $m=3$, the adaptive strategy leads to convergence in 53 iterations, and starting with $m=5$ gives convergence in 58 iterations.  Convergence results with constant $m$=3,5,10 are also shown for comparison (these are the same  as in Figure \ref{B2} at bottom left), and we observe that the simple adaptive strategy gives a significant speedup in convergence.

\subsection{Incompressible Navier-Stokes equations}

For our third test problem, we consider the incompressible NSE and the channel expansion problem studied by Farrell et al \cite{FBF15} with domain $\Omega=(0,2.5)\times (-1,1) \cup (2.5,150)\times(-6,6)$ representing a narrow channel of width 2 entering into a channel of width 12.  The NSE takes the form:
	\begin{equation}\label{NS1}
 		\left\{\begin{aligned}
 			-\nu \Delta u+u\cdot\nabla u+ \nabla p&=0 \quad \text{in}~\Omega,\\
 			\nabla\cdot {u}&=0\quad \text{in}~\Omega,
 		\end{aligned}\right.
 	\end{equation}
where $u$ and $p$ represent the unknown velocity and pressure, respectively, and $\nu>0$ is the kinematic viscosity which for this test is the inverse of the Reynolds number $Re = \nu^{-1}.$  Dirichlet boundary conditions are enforced as parabolic inflow with max inlet velocity of 1 at $x=0$ and as no-slip on the walls, and zero traction is weakly enforced at the outflow.  We choose $Re$=50 for this test, for which the NSE is known to have 5 distinct solutions \cite{FBF15}, although two of these are symmetric reflections of others (with one solution being self-reflectionally symmetric).  The 5 distinct solutions are shown in Figure \ref{N1}.

State of the art techniques for finding multiple solutions typically use variations of (deflated) Newton-type continuation methods \cite{FBF15}, and so one must start with very low $Re$ and find all solutions there before `climbing the ladder' to slightly higher $Re$, repeating the process until all solutions are found at the desired $Re$.  While generally effective, this process can be slow since many solves need done at lower $Re$.  Hence we aim to find solutions directly at $Re$=50 using NGMRES-Picard, and avoid continuation methods altogether.

\begin{figure}[ht!]
\begin{center}
\includegraphics[width = .8\textwidth, height=.16\textwidth,viewport=170 20 1300 250, clip]{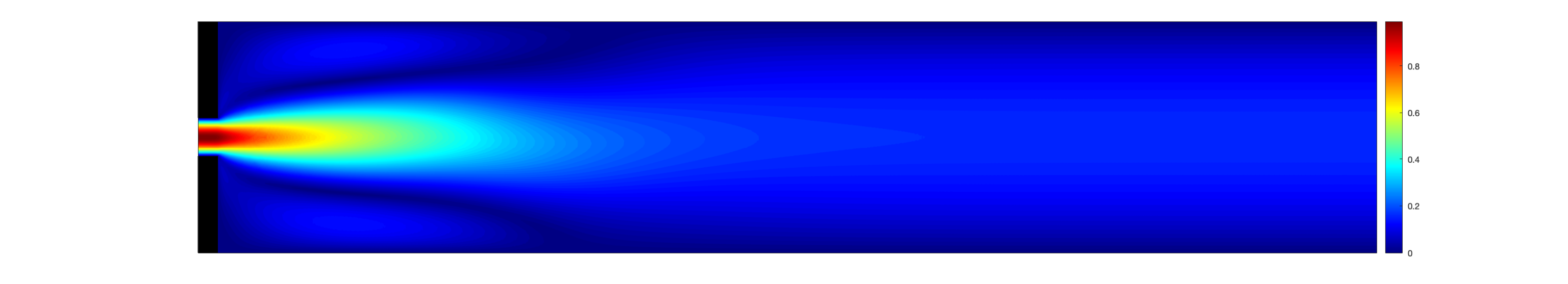}\\
\includegraphics[width = .8\textwidth, height=.16\textwidth,viewport=170 20 1300 250, clip]{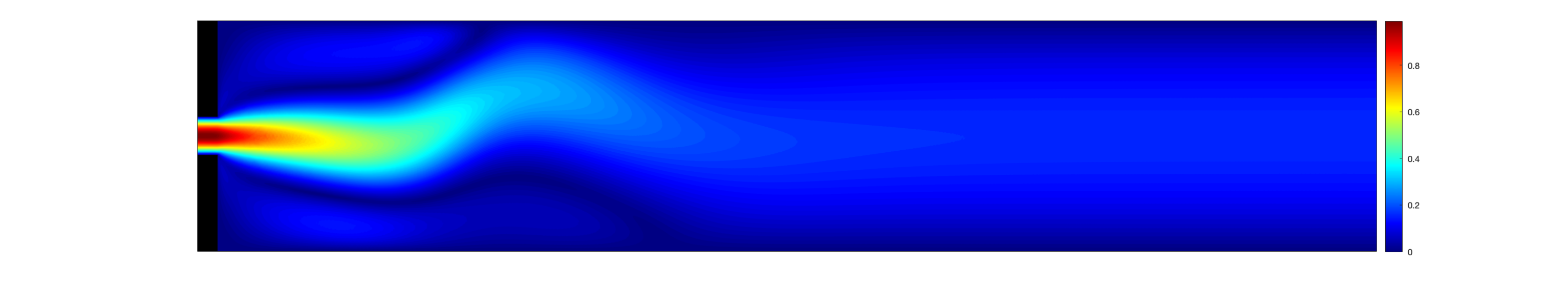}\\
\includegraphics[width = .8\textwidth, height=.16\textwidth,viewport=170 20 1300 250, clip]{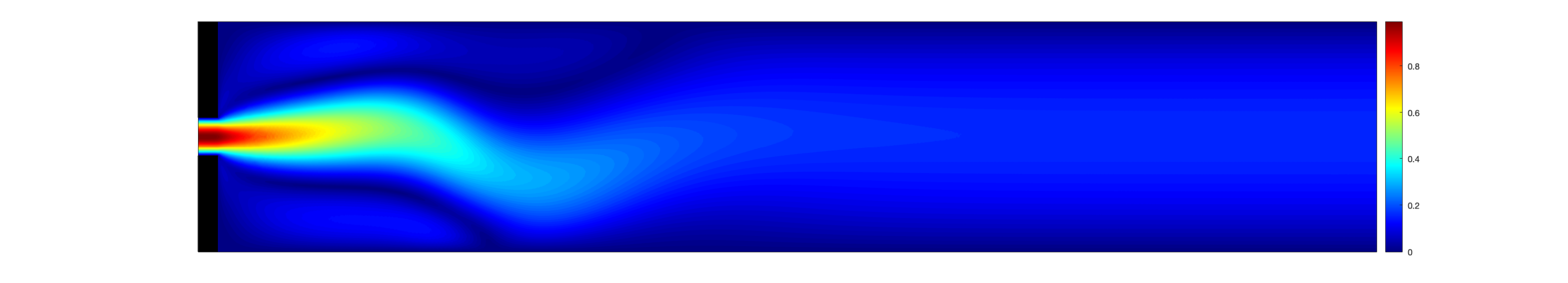}\\
\includegraphics[width = .8\textwidth, height=.16\textwidth,viewport=170 20 1300 250, clip]{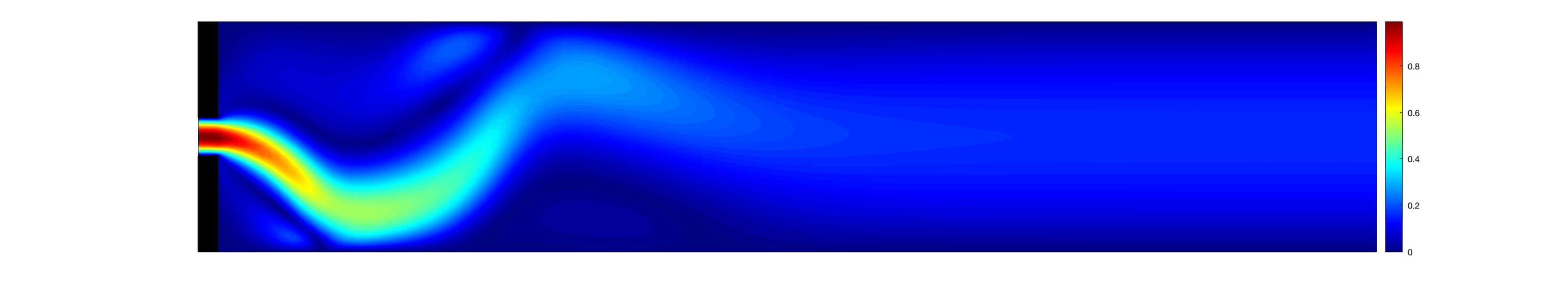}\\
\includegraphics[width = .8\textwidth, height=.16\textwidth,viewport=170 20 1300 250, clip]{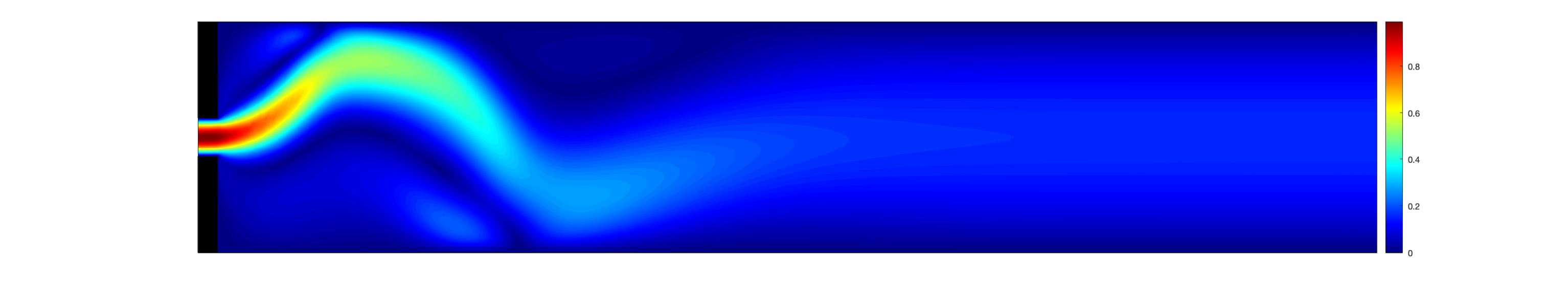}
\caption{\label{N1}  Shown above are the 5 solutions found by the NGMRES deflation method for the NSE channel expansion problem. }
\end{center}
\end{figure}

The Picard iteration for the NSE is given by 
 	\begin{equation}\label{Pic1}
		\left\{\begin{aligned}
 		-\nu \Delta u_{k+1}+u_k\cdot\nabla u_{k+1}+ \nabla p_{k+1}&=0, \\
 		\nabla\cdot {u}_{k+1}&=0,
 		\end{aligned}\right.
		\end{equation}	
together with the boundary conditions given above.  We now use NGMRES-Picard together with deflation and to directly find all five solutions (and note it is enough to find 3 solutions not reflectionally symmetric of each other).

To use NGMRES-Picard to directly find still unknown solutions, we simply rescale the NGMRES nonlinear residual at each step by
\[
\| g(w) \|_{V'} \rightarrow \frac{ \| g(w) \|_{V'}}{\sum_{j=1}^n \| w - u_i \|_{V}^3  },
\]
where $V$ is the divergence-free velocity space and $V'$ is its dual, and $n$ is the number of velocity solutions $\{ u_i \}_{i=1}^n$ already found.  Such a deflation rescaling is suggested in \cite{FBF15}, and we find for this problem the cubic power works well.  The general idea is that if the iteration gets close to $u_i$, the residual will grow and push the iteration away through the optimization problem.  

Hence the algorithm is to run NGMRES-Picard until a solution is found, and then rerun using the adjusted nonlinear residual until all solutions are found.  Initial guesses are random perturbations of the Stokes solution.  We set a NGMRES-Picard nonlinear residual convergence tolerance to $10^{-5}$ in the $V'$ norm, and this solution is then passed to Newton to quickly zoom in on the solution.  We used $m=10$ with restarts every 30 iterations, and 100 max iterations per run.

\begin{figure}[ht!]
\begin{center}
\includegraphics[width = .95\textwidth, height=.2\textwidth,viewport=250 0 1900 250, clip]{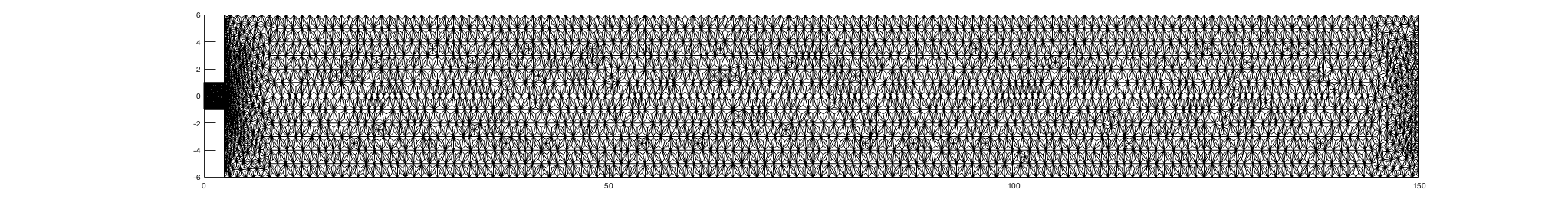}\\
\includegraphics[width = .5\textwidth, height=.4\textwidth,viewport=250 0 400 250, clip]{meshexp.png}
\caption{\label{N2}  Shown above is the mesh used for the NGMRES-Picard deflation method for the NSE channel expansion problem, as the entire domain (top) and zoomed in to the inlet (bottom). }
\end{center}
\end{figure}

We discretize using $(P_2,P_1^{disc})$ SV elements on a barycenter refinement of a Delaunay mesh refined more heavily near the inlet.  A plot of the mesh is shown in Figure \ref{N2}, and this discretization provides for 97K total dof.  Since the system is small, we use direct solves for the linear systems.

The procedure works well overall, and was able to find all 3 solutions that are not reflectionally symmetric of each other in 4 total runs.  The first run found the symmetric solution quickly, converging in 45 iterations, the second run found the solution at the bottom in Figure \ref{N1} in 69 iterations, the third run did not converge in 100 iterations, and the fourth run found the solution second from top in Figure \ref{N1} in 67 iterations.  Undoubtedly, our algorithm could be improved e.g by optimizing NGMRES parameters and restarts, allowing more than 100 iterations at each run, or by using smarter Newton algorithms (we used classical, unstabilized Newton), but still NGMRES was able to easily incorporate deflation and was quite effective.

\subsection{NGMRES for superlinearly converging iterations}
For a final test, we illustrate the potential for $m=0$ NGMRES to be used with superlinearly converging iterations.  NGMRES (and AA) is likely to help a superlinear solver at early iterations, before the iterates are close enough for superlinear convergence to kick in.  Once the iterates get close to a root, how NGMRES affects convergence is not completely understood.  However, our theory above in Theorem \ref{thm:ngmres0} shows that the higher order terms are quadratic, which suggests that $m=0$ NGMRES will not slow down an order $1<r\le 2$ iteration. While this is not a proof, NGMRES theory does compare well with depth 1 AA theory, whose fixed point residual bound has higher order terms that are less than quadratic, and furthermore AA is known to reduce convergence order $r>1$ to $\frac{r+1}{2}$ \cite{X23,RX23}.

We now test convergence of $m=0$ NGMRES and depth 1 AA (using the standard AA from e.g. \cite{PR25}) on Newton and Secant iterations to solve $g(x) = x^2 -x-2 =0$:
\[
q_N(x) =  x - \frac{g(x)}{g'(x)}, \ \  
q_S([x,y]) = [x - \frac{g(x)}{ \frac{g(x)-g(y)}{x-y}},x].
\]
Using a 1D test allows us to focus on the higher order term effects.

Results are shown in Figure \ref{S1} for initial guess $[0,100]$ and stopping tolerance of $10^{-10}$ for the nonlinear residual.  For both tests, we observe that NGMRES improves convergence, while AA slows it down.  The improvement from NGMRES comes early in the iteration, as expected, but NGMRES does not appear to slow down the superlinear asymptotic convergence rates of the Newton and Secant iterations.  While this is a simple 1D test, it suggests that further study of NGMRES applied to superlinearly converging methods is warranted.

\begin{figure}[h!]
\centering
\includegraphics[scale = 0.35]{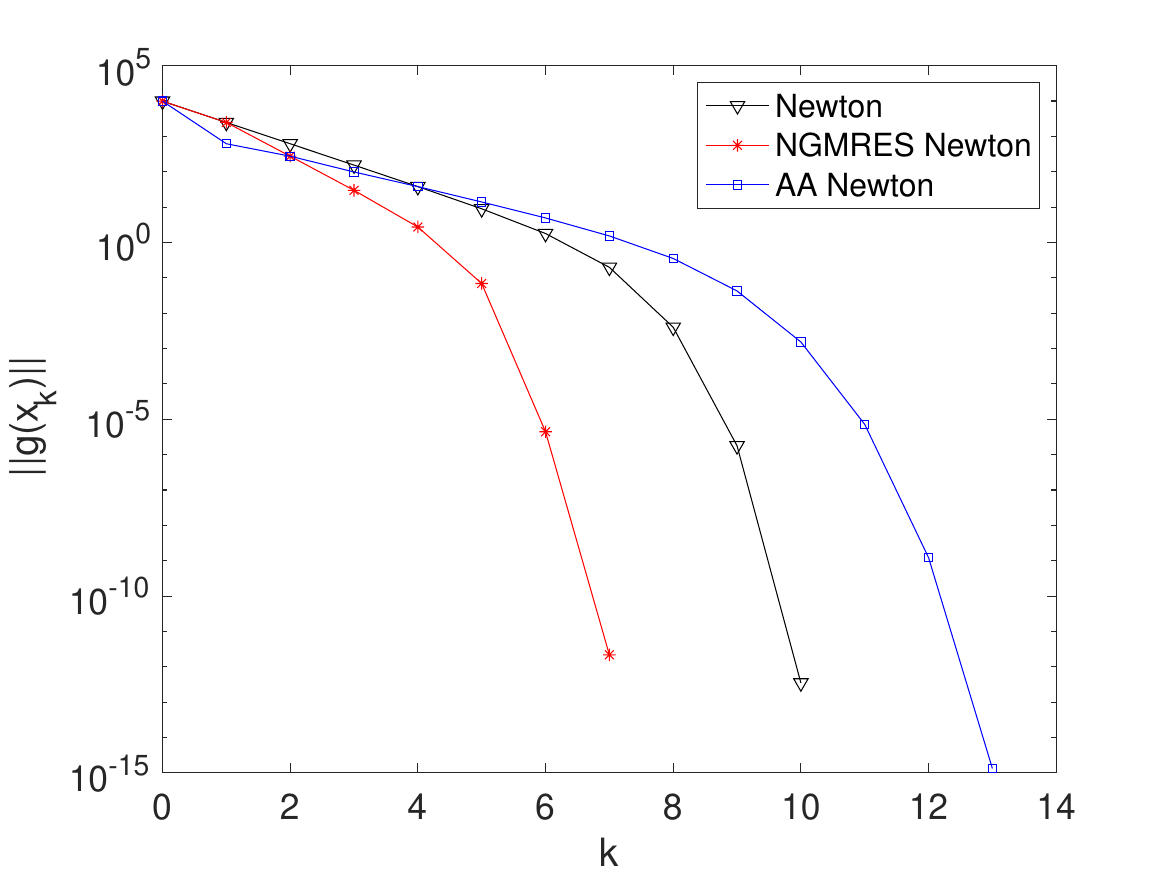}
\includegraphics[scale = 0.35]{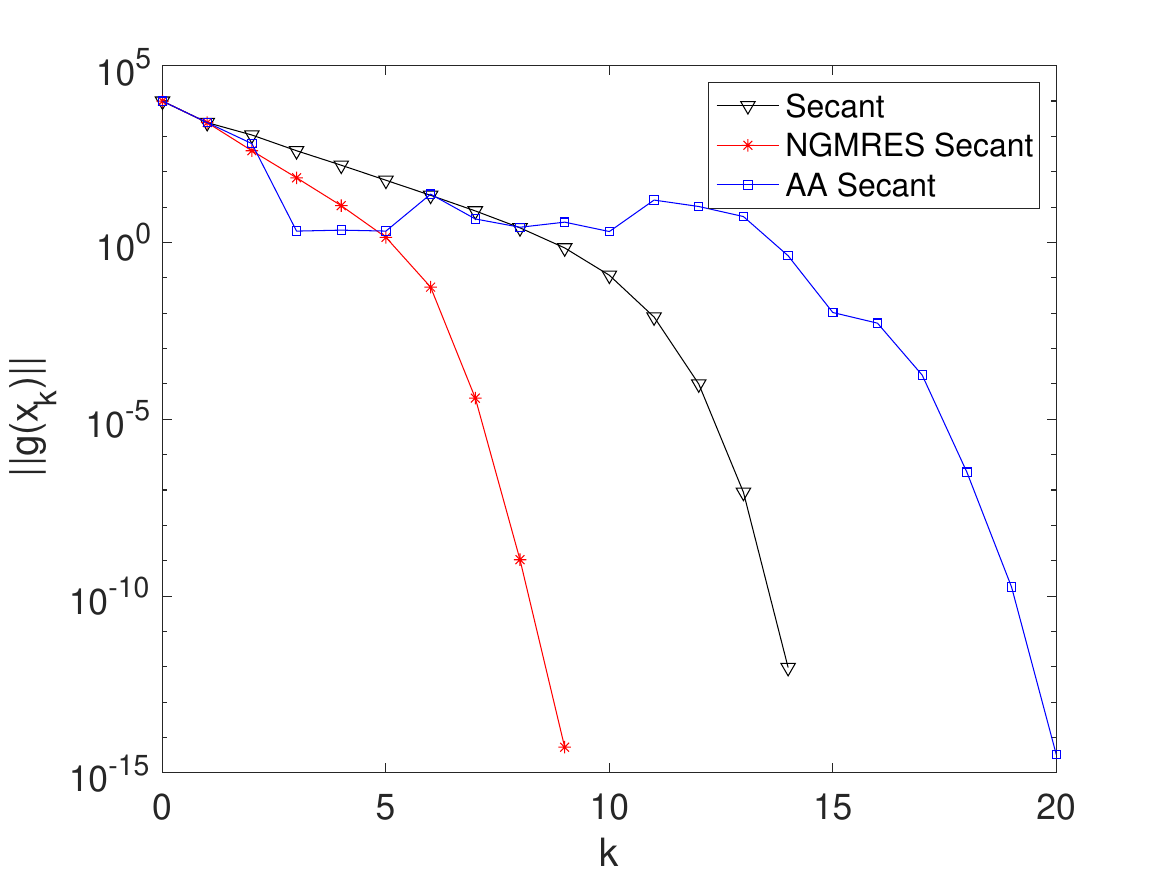}
\caption{\label{S1} Shown above are convergence plots for Newton and Secant iterations, along with NGMRES and AA applied.}
\end{figure}

\section{Conclusions}

This paper has provided the first proof of NGMRES acceleration for contractive and noncontractive fixed point iterations used to solve general nonlinear systems.  We have established the gain of the optimization problem as the mechanism responsible for the acceleration, and have also identified an additional quantity arising from the optimization problem that accurately predicts the linear convergence rate of the nonlinear residuals at each step.  The newly identified quantity, since it is at most 1, also shows the NGMRES only reduces the nonlinear residual (up to higher order terms).  Several numerical tests illustrate the theory and also show how NGMRES can be improved for certain problems by using restarts and/or an adaptive depth strategy based on our newly discovered linear convergence rate predictor .  Additionally, we have shown that NGMRES has the potential to improve superlinearly converging iterations (an advantage over AA).

For future work, we believe there are several important directions.  First, further improvements can be made to the adaptive depth strategy given above, e.g. by considering what is the best depth between $1$ and $m$ to use at any given iteration, based on weighing higher order term effects against improvement in the linear convergence rate.  Second, NGMRES appears to help with superlinearly converging iterations (as opposed to AA, which makes them worse); this should be explored both analytically and numerically.  Third, more documented comparisons of AA and NGMRES need performed to better classify when one is better than the other; the analysis herein may give some hints in this direction.  Finally, it may be possible that our result for the noncontractive case with $m\ge 1$ may be improvable in the higher order terms, since our NLH tests found convergence sometimes after 500 iterations.


\end{document}